\documentclass[12pt]{amsart}
\usepackage{ifthen,verbatim}
\usepackage{floatflt,graphicx,graphics}

\usepackage{tikz} 
\usepackage{subcaption}
\usepackage{pdfpages}
\usepackage{mathrsfs}
\usepackage{amsmath,amssymb,amsthm}
\usepackage{mathtools}
\usepackage{multirow}

\numberwithin{equation}{section}
\nonstopmode
\numberwithin{equation}{section}
\theoremstyle{plain}
\newtheorem{thm}[equation]{Theorem}
\newtheorem{cor}[equation]{Corollary}
\newtheorem{lem}[equation]{Lemma}

\newtheorem{prop}[equation]{Proposition}

\theoremstyle{definition}

\newtheorem{rem}[equation]{Remark}
\theoremstyle{remark}

\newtheorem{nonsec}[equation]{}

\numberwithin{equation}{section}

\newcommand{\beq}{\begin{equation}}
\newcommand{\eeq}{\end{equation}}
\newcommand{\ben}{\begin{enumerate}}
\newcommand{\een}{\end{enumerate}}
\newcommand{\bequu}{\begin{eqnarray*}}
\newcommand{\eequu}{\end{eqnarray*}}
\newcommand{\bequ}{\begin{eqnarray}}
\newcommand{\eequ}{\end{eqnarray}}

\theoremstyle{remark}

\newcommand{\sh}{\,\textnormal{sh}}

\renewcommand{\th}{\,\textnormal{th}}
\renewcommand{\Im}{{ \rm Im}\,}

\font\fFt=eusm10 
\font\fFa=eusm7  
\font\fFp=eusm5  
\def\K{\mathchoice
{\hbox{\,\fFt K}}
{\hbox{\,\fFt K}}
{\hbox{\,\fFa K}}
{\hbox{\,\fFp K}}}

\newcounter{alphabet}

\newcounter{minutes}
\divide\time by 60
\newcounter{hours}
\multiply\time by 60
\addtocounter{minutes}{-\time}

\begin{document}
\bibliographystyle{amsplain}
\title{Visual angle metric in the upper half plane}

\def\thefootnote{}
\footnotetext{
\texttt{\tiny File:~\jobname .tex,
          printed: \number\year-\number\month-\number\day,
          \thehours.\ifnum\theminutes<10{0}\fi\theminutes}
}
\makeatletter\def\thefootnote{\@arabic\c@footnote}\makeatother

\author[M. Fujimura]{Masayo Fujimura}
\author[O. Rainio]{Oona Rainio}
\author[M. Vuorinen]{Matti Vuorinen}

\date{}


\keywords{M\"obius transformations, hyperbolic metric, visual angle metric, quasiconformal mappings}
\subjclass[2020]{Primary 51M09; Secondary 30C62}
\begin{abstract}
We prove an identity which connects the visual angle metric $v_{\mathbb{H}^2}$ and the hyperbolic metric $\rho_{\mathbb{H}^2}$ of the upper half plane $\mathbb{H}^2$. The proof is based on geometric arguments and uses computer algebra methods for formula manipulation. We also prove a sharp Hölder continuity result for quasiregular mappings with respect to the visual angle metric.
\end{abstract}
\maketitle

\noindent \textbf{Author information.}\\
Masayo Fujimura$^1$, email: \texttt{masayo@nda.ac.jp}, ORCID: 0000-0002-5837-8167\\
Oona Rainio$^2$, email: \texttt{ormrai@utu.fi}, ORCID: 0000-0002-7775-7656\\
Matti Vuorinen$^2$, email: \texttt{vuorinen@utu.fi}, ORCID: 0000-0002-1734-8228\\
1: Department of Mathematics, National Defense Academy of Japan, Yokosuka, Japan\\
2: Department of Mathematics and Statistics, University of Turku, FI-20014 Turku, Finland\\
\textbf{Funding.} The second author was supported by Finnish Culture Foundation.\\
\textbf{Data availability statement.} Not applicable, no new data was generated.\\
\textbf{Conflict of interest statement.} There is no conflict of interest.\\



  
\section{Introduction}
  
During the past few decades, many authors have applied metrics as a 
tool of geometric
function theory research \cite{gh,ha,h, hkv,p}. This research is 
motivated by the need to find such metrics in general planar domains that have a more explicit expression than the hyperbolic metric. These metrics typically are not M\"obius invariant, but they are often quasi-invariant and within a constant factor from the hyperbolic metric.
Here, we study one of these metrics, the visual angle metric, 
introduced in \cite{klvw}
and further investigated in \cite{hvw,wv,fkv}.

For a given domain
$G \subset\mathbb{R}^2$ and $a,b \in G,$ the visual angle metric 
maximizes the visual
angle $\measuredangle(a,z,b)$ over all boundary points $z \in \partial G$ and 
it is therefore defined as follows:
\begin{equation*}
  v_G(a,b) = \sup \{ \measuredangle (a,z,b) : z\in \partial{G}\}\, .
\end{equation*}
To avoid trivial cases, we require that $G$ is a proper subdomain of $\mathbb{R}^2$ such that $\partial G$
is not a proper subset of a line. In spite of this very concrete definition, the visual angle metric does 
not have
an explicit formula but several comparison inequalities between the 
visual angle metric
and other metrics are known \cite{klvw,hvw,wv}.

We continue the work of \cite{fkv}, where an explicit formula for the 
visual angle metric involving the hyperbolic
metric was found in the case when the domain is the unit disk $\mathbb{B}^2.$
It seems natural to expect that finding a similar formula for the upper 
half plane $\mathbb{H}^2$ should be possible. The fundamental 
difficulty here is that the hyperbolic
metric is invariant under  M\"obius automorphisms of $\mathbb{H}^2$
while  the visual angle metric is not. Moreover, we were  not
able to extend the proof of  \cite{fkv} to the upper half plane case. 
Our first main result is the following theorem where we give a formula
for the visual angle metric $v_{\mathbb{H}^2}$ of the upper half plane,
in terms of the hyperbolic metric $\rho_{\mathbb{H}^2}.$ 
Due to the invariance of both $\rho_{\mathbb{H}^2}$ and $v_{\mathbb{H}^2}$
under similarity transformations, we give the result for normalized point
pairs  $ a,b\in S^1(0,1)\cap\mathbb{H}^2 .$
%

\begin{thm} \label{my1} For $ a,b\in S^1(0,1)\cap\mathbb{H}^2 $, let $u=(1+ab)/(a+b),\,$
\[
T\equiv \frac{(1+\sqrt{1-|u|^2})\, t \sqrt{1-t^2}}{\sqrt{1- |u|^2 t^2}} 
,\quad
t= \frac{|a-b|}{|a-\overline{b}|}= \frac{|a-b|}{|ab-1|}\, \,. \]
Then
\[
\sin(v_{\mathbb{H}^2}(a,b))= T  =\frac{|a-b|}{|ab-1|}\left(
    \frac12|a+b|+\sqrt{{\rm Im}(a)\cdot {\rm Im}(b)}\right).
\]

\end{thm}

The proof is based on the method of the earlier papers \cite{fkv, frv},
and it utilizes computer algebra algorithms combined with manual processing. 
In particular, we give  explicit formulas for the points of 
intersection of circles and lines in Euclidean and hyperbolic geometries, needed for the proofs. Some of challenges of symbolic computation are briefly discussed in \ref{challenges}. The formulas are also simple enough
to establish the connection between the visual angle metric and the 
hyperbolic metric (see Figures \ref{fig3} and \ref{fig4}). To this end, 
we prove a number of geometric lemmas on which
the main steps of the proof of Theorem \ref{my1} are based. There are four steps that
are described in the beginning of Section 5 and in Figure \ref{fig7}.

Our second main result is the following sharp H\"older continuity result for the visual angle metric.

\begin{thm}\label{thm_holder}
If $f:\mathbb{H}^2\to\mathbb{H}^2= f(\mathbb{H}^2)$ is a non-constant $K$-quasiregular mapping, then for all $a,b\in\mathbb{H}^2$ with $v_{\mathbb{B}^2}(a,b) < \pi/2$, 
\begin{align*}
\tan\left(\frac{v_{\mathbb{H}^2}(f(a),f(b))}{2}\right)\leq\lambda(K)^{1/2}\max\left\{\left(\tan(v_{\mathbb{H}^2}(a,b))\right)^K,\left(\tan(v_{\mathbb{H}^2}(a,b))\right)^{1/K}\right\}.    
\end{align*}
\end{thm}

The structure of this article is as follows. First, we present preliminary results related to Euclidean and hyperbolic geometries. In Section 3, we find necessary formulas for intersection points related to the distance $v_{\mathbb{H}^2}$. These formulas, proved with computer algebra methods and manual postprocessing, are given in Tables \ref{t1} and \ref{t2}. The formulas are simple and explicit and may be of independent interest. In Section 4, we prove that certain points related to our main result are collinear. The first main result, Theorem \ref{my1}, is proven in Section 5. In Section 6, we present our second main result, Theorem \ref{thm_holder}, about the H\"older continuity of $K$-quasiregular mappings with respect to the visual angle metric $v_{\mathbb{H}^2}$. This result is sharp when $K=1$.

%


  
\section{Preliminary results}

Let $L[a,b]$ stand for the line through $a$ and $b\,(\neq a)$.
For distinct points $a,b,c,d \in {\mathbb{C}}$ such that the lines
$L[a,b]$ and  $L[c,d]$ intersect at a unique
point $w$, let
$$
w=LIS[a,b,c,d] = L[a,b]\cap L[c,d] \,.
$$
This point is given by (see e.g. \cite[Ex. 4.3(1), p. 57 and p. 373]{hkv})
\begin{equation}{\label{LIS}}
w=LIS[a,b,c,d] =\frac{(\overline{a}b -a  \overline{b})(c-d)-(a-b) ( \overline{c}d -c  \overline{d})}{( \overline{a}- \overline{b})(c-d)-(a-b) (\overline{c} - \overline{d})}\,.
\end{equation}

Let $C[a,b,c]$  be the  circle  through distinct noncollinear points $a$, $b$,
and $c$. The formula \eqref{LIS} gives easily a formula for the center
$m(a,b,c)$ of $C[a,b,c]\,.$ For instance, we can find two points on
the bisecting normal to the side $[a,b]$ and another
two points on the bisecting normal to the side $[a,c]$ and then apply
\eqref{LIS}  to get $m(a,b,c)\,.$
In this way we see that the center $m(a,b,c)$ of  $C[a,b,c]$  is
\begin{equation}\label{mfun}
\begin{aligned}
m(a,b,c)&=LIS[\frac{a+b}{2}, \frac{a+b}{2}+(a-b)i,\frac{a+c}{2},\frac{a+c}{2}+(a-c)i]\\
&=\frac{ |a|^2(b-c) +  |b|^2(c-a) +  |c|^2(a-b) }
{a(\overline{c}-\overline{b}) +b(\overline{a}-\overline{c})+ 
c(\overline{b}-\overline{a})}\,.    
\end{aligned}
\end{equation}

We denote the Euclidean ball with a center $x\in\mathbb{R}^n$ and a radius $r>0$ by $B^n(x,r)=\{y\in\mathbb{R}^n\text{ : }|x-y|<r\}$ and the corresponding boundary sphere by $S^{n-1}(x,r)=\{y\in\mathbb{R}^n\text{ : }|x-y|=r\}$. We use the simplified notation $\mathbb{B}^n$ for the unit ball $B^n(0,1)$ and $S^{n-1}$ for the unit sphere. The upper half space $\mathbb{H}^n$ is the domain $\{x=(x_1,...,x_n)\in\mathbb{R}^n\text{ : }x_n>0\}$.

The chordal metric is defined
for $a,b \in \overline{\mathbb{R}^2} =\mathbb{R}^2 \cup \{\infty\} $ as

\begin{equation}\label{chordal}
  \begin{cases}
     {\displaystyle
			q(a,b)=\frac{|a-b|}{\sqrt{1+|a|^2}\;\sqrt{1+|b|^2}}}\;;\,\;a\ne\infty\ne b\;,&\\
			{\displaystyle q(a,\infty)=\frac{1}{\sqrt{1+|a|^2}}}\;\,.&
  \end{cases}
\end{equation}
The \emph{absolute ratio}, also known as the \emph{cross-ratio}, for any four distinct points $a,b,c,d\in\overline{\mathbb{R}}^n$ can now be defined as \cite[(3.10), p. 33]{hkv}, \cite{b}
\begin{align}
|a,b,c,d|=\frac{q(a,c)q(b,d)}{q(a,b)q(c,d)}.
\end{align}
The chordal distances above can be replaced with the Euclidean distances if $\infty\notin\{a,b,c,d\}$.

\begin{nonsec}{\bf Hyperbolic geometry.}\label{hg}
We recall some basic formulas and notation for hyperbolic geometry from \cite{b}.

\end{nonsec}

The hyperbolic metrics of the unit disk ${\mathbb{B}^2}$ and
the upper half plane $\mathbb{H}^2$ are defined, resp., by
\begin{equation}\label{rhoB}
\sh \frac{\rho_{\mathbb{B}^2}(a,b)}{2}=
\frac{|a-b|}{\sqrt{(1-|a|^2)(1-|b|^2)}} ,\quad a,b\in \mathbb{B}^2\,,
\end{equation}
and
\begin{equation}\label{rhoH}
{\rm th}\frac{\rho_{\mathbb{H}^2}(a,b)}{2}=
\frac{|a-b|}{|a-\overline{b}|} ,\quad a,b\in \mathbb{H}^2.
\end{equation}
Here  ${\rm sh}$ and ${\rm th}$ stand for the hyperbolic sine and tangent,  
and their inverse functions are ${\rm arsh}$ and ${\rm arth},$ respectively.
These two metrics are M\"obius invariant and also connected by the identity
\begin{equation}\label{rhoHB}
\rho_{\mathbb{H}^2}(a,b) = 2 \rho_{\mathbb{B}^2}({\rm Re}( a),{\rm Re}(b)), \quad
a,b \in S^1(0,1) \cap \mathbb{H}^2.
\end{equation}
Moreover, for $a,b \in S^1(0,1) \cap \mathbb{H}^2$, the point $u=LIS[a, \overline{b},b, \overline{a}]$
satisfies (cf. \cite[Ex. 3.21(2)]{hkv}, \cite[Fig. 10]{vw})
\begin{equation} \label{hmid}
\rho_{\mathbb{B}^2}({\rm Re}( a),u)=\rho_{\mathbb{B}^2}({\rm Re}( b),u),
\end{equation}
i.e. $u$ is the hyperbolic midpoint of ${\rm Re}( a)$ and ${\rm Re}( b).$ In particular,
for $r\in(0,1)$
\begin{equation} \label{hmid2}
2\rho_{\mathbb{B}^2}(0,d)=\rho_{\mathbb{B}^2}(0,r)\,, 
\quad d= \frac{r}{1+\sqrt{1-r^2}}\,.
\end{equation}


\begin{nonsec}{\bf Geodesics of the hyperbolic metric.  
}\label{epFtex}  
\label{geod1}
For $ a,b\in\mathbb{H}^2 $, the geodesic line of the hyperbolic metric
 through the points is the upper semicircle through the points and with the 
 center at the point $LIS[0,1,(a+b)/2, i (a-b)]$ of the real axis (see \cite{b}.)
We now find the points of intersection of the real axis and the full circle 
$ C $ passing through
$ a,b,\overline{a} .$  

\bigskip
Assuming $  {\rm Im}\, a\neq 0 $,
the center of $ C $ is given by \eqref{mfun} as
\begin{align*}
 m(a,b,\overline{a}) 
   &= \frac{|a|^2(b-\overline{a})+|b|^2(\overline{a}-a)+|\overline{a}|^2(a-b)}
           {a(a-\overline{b})+b(\overline{a}-a)
               +\overline{a}(\overline{b}-\overline{a})}\\
   &=\frac{a\overline{a}(a-\overline{a})-b\overline{b}(a-\overline{a})}
          {(a^2-\overline{a}^2)-(b+\overline{b})(a-\overline{a})}
    =\frac{a\overline{a}-b\overline{b}}{(a+\overline{a})-(b+\overline{b})}.
\end{align*}
The radius $ r $ of $ C $ is
$$
  r=\Big|\frac{a\overline{a}-b\overline{b}}{(a+\overline{a})-(b+\overline{b})}
        -a\Big| 
   =\Big|\frac{-a(a-b)+\overline{b}(a-b)}
              {(a+\overline{a})-(b+\overline{b})}\Big|.
$$
Hence, $ C $ is the circle defined by
$$
 \left\{ z: \Big|z-\frac{a\overline{a}-b\overline{b}}
              {(a+\overline{a})-(b+\overline{b})}\Big| =\Big|\frac{-a(a-b)+\overline{b}(a-b)}
              {(a+\overline{a})-(b+\overline{b})}\Big| \right\}.
$$
Let $ x $  be the intersection point of $ C $ and the real axis.
Then, $ x $ is a solution to
$$
     |x\big((a+\overline{a})-(b+\overline{b})\big)
         -(a\overline{a}-b\overline{b})|
    =|-a+\overline{b}||a-b|.
$$
So,
\begin{align*}
  \Big(2x\big( {\rm Re}\,(a-b)\big)-(a\overline{a}-b\overline{b})\Big)^2
   & = (-a+\overline{b})(-\overline{a}+b)(a-b)(\overline{a}-\overline{b})\\
  \Big(2x\big( {\rm Re}\,(a-b)\big)-(a\overline{a}-b\overline{b})\Big)^2
   & =(a\overline{a}+b\overline{b}-ab-\overline{a}\overline{b})
      (a\overline{a}+b\overline{b}-a\overline{b}-\overline{a}b)\\
  4x^2\big( {\rm Re}\,(a-b)\big)^2
    -4x {\rm Re}\,(a-b)(a\overline{a}-b\overline{b})
    & =(ab\overline{b}+\overline{a}b\overline{b}-a\overline{a}b
       -a\overline{a}\overline{b})(a+\overline{a}-b-\overline{b}).
\end{align*}
Simplifying the equation further, we obtain
\begin{equation}\label{eq:astar1}
   {\rm Re}\,(a-b)x^2-(|a|^2-|b|^2)x
     +\big(|a|^2 {\rm Re}\,(b)-|b|^2 {\rm Re}\,(a)\big)=0.
\end{equation}
  \end{nonsec}

The equation \eqref{eq:astar1} or the equivalent one
\begin{equation}\label{eq:aster2}
  (a-b+\overline{a}-\overline{b})x^2-2(a\overline{a}-b\overline{b})x
     +\big(a\overline{a}(b+\overline{b})-b\overline{b}(a+\overline{a})\big)=0
\end{equation}
has two solutions.
These two solutions are
\begin{align*}
  ep^+ &=\frac{|a|^2-|b|^2+|a-b||a-\overline{b}|}{2 \, {\rm Re}\,(a-b)}
      =\frac{a\overline{a}-b\overline{b}+|a-b||a-\overline{b}|}
       {a-b+\overline{a}-\overline{b}},\\
  ep^- &=\frac{|a|^2-|b|^2-|a-b||a-\overline{b}|}{2\,  {\rm Re}\,(a-b)}
     =\frac{a\overline{a}-b\overline{b}-|a-b||a-\overline{b}|}
       {a-b+\overline{a}-\overline{b}}.
\end{align*}

\begin{lem} \label{epF} For  $ a,b\in\mathbb{H}^2 $, the two points
$$
  a_{\ast}=ep^+  =\frac{|a|^2-|b|^2+|a-b||a-\overline{b}|}{2 \,{\rm Re}\,(a-b)} \,, \quad
  b_{\ast}= ep^- = \frac{|a|^2-|b|^2-|a-b||a-\overline{b}|}{2 \,{\rm Re}\,(a-b)} \,,
        $$
are the end points of the semicircle through $a$ and $b,$  centered
at a point of the real axis, and the points occur in
the order $ a_{\ast} , a,b,b_{\ast}$ on the semicircle.        
\end{lem}

\begin{proof}

In the case of $  {\rm Re}\,a> {\rm Re}\,b $, we have 
 $ a_{\ast}=\max\{ep^+,ep^-\} $.
Since $  {\rm Re}\,(a-b)>0 $,
$$
  ep^-=\frac{|a|^2-|b|^2-|a-b||a-\overline{b}|}{2\,  {\rm Re}\,(a-b)}
      <\frac{|a|^2-|b|^2+|a-b||a-\overline{b}|}{2\, {\rm Re}\,(a-b)}
      =ep^+
$$
holds and $ a_{\ast}=ep^+ $.

In the other case, $  {\rm Re}\,a< {\rm Re}\,b $, we have 
 $ a_{\ast}=\min\{ep^+,ep^-\} $.
Since $  {\rm Re}\,(a-b)<0 $,
$$
  ep^+=\frac{|a|^2-|b|^2+|a-b||a-\overline{b}|}{2\,  {\rm Re}\,(a-b)}
      <\frac{|a|^2-|b|^2-|a-b||a-\overline{b}|}{2\,  {\rm Re}\,(a-b)}
      =ep^-
$$
holds and $ a_{\ast}=ep^+ $.
In both cases, we obtain $ a_{\ast}=ep^+ $. 
\end{proof}

The points $a_{\ast}$ and  $b_{\ast}$ are called the {\it  endpoints} of the hyperbolic
line through $a$ and $b.$ 
\bigskip

\begin{lem} \label{quadr} For  $ a,b\in\mathbb{H}^2 $, the cross-ratio is
$$
 \big| a_{\ast},a,b,b_{\ast}\big|
   =\frac{|a-\overline{b}|+|a-b|}{|a-\overline{b}|-|a-b|}.
$$
\end{lem}

\bigskip


\begin{proof}

Substituting $ a_{\ast}=ep^+ $ and $ b_{\ast}=ep^- $ into
$$ 
  \big| a_{\ast},a,b,b_{\ast}\big|
   =\frac{|a_{\ast}-b||a-b_{\ast}|}{|a_{\ast}-a||b-b_{\ast}|}
$$
and simplifying, we have
$$
  \big| a_{\ast},a,b,b_{\ast}\big|
   =\bigg|\frac{\big((b-\overline{a})(a-b)-|a-b||a-\overline{b}|\big)
            \big((a-\overline{b})(a-b)+|a-b||a-\overline{b}|\big)}
         {\big((b-\overline{a})(a-b)+|a-b||a-\overline{b}|\big)
           \big((a-\overline{b})(a-b)-|a-b||a-\overline{b}|\big)}\bigg|.
$$
As $|a-b|^2|a-\overline{b}|^2
  =(a-b)(\overline{a}-\overline{b})(a-\overline{b})(\overline{a}-b) $,
we obtain
\begin{align*}
 \big| a_{\ast},a,b,b_{\ast}\big|
   &=\bigg|\frac{-(a-\overline{b})(\overline{a}-b)-|a-b||a-\overline{b}|}
          {-(a-\overline{b})(\overline{a}-b)+|a-b||a-\overline{b}|}\bigg| \\
   &=\bigg|\frac{|a-\overline{b}|+|a-b|}{|a-\overline{b}|-|a-b|}\bigg|.
\end{align*}
By the assumption $ a,b\in\mathbb{H}^2 $, the inequality
$ |a-b|<|a-\overline{b}| $ holds, and the assertion is obtained.
\end{proof}

The next proposition is a well-known basic fact \cite{b}.
\bigskip

\begin{prop} \label{rhoabsr} For  $ a,b\in\mathbb{H}^2 $
$$
   \rho_{\mathbb{H}^2}(a,b)=\log\big|a_{\ast},a,b,b_{\ast}\big|.
$$
\end{prop}
\bigskip

\begin{proof}
The assertion is immediately obtained from the above Lemma \ref{quadr}.
\end{proof}

\begin{nonsec}{\bf Normalization of point pairs.}
For $a,b \in \mathbb{H}^2,$ the above results yield  formulas for $u_1\equiv(a_*+b_*)/2$
and $\lambda\equiv 1/|a-u_1|.$ By utilizing these formulas and the similarity transformation
$$
z \mapsto  \lambda(z-u_1)\,,\quad z\in \mathbb{H}^2\,,
$$
which maps $ \mathbb{H}^2$ onto itself,  we see that $a$ and $b$ are mapped into
$S^1(0,1).$ 
Obviously, $v_{ \mathbb{H}^2}(a,b) $ is invariant under this transformation and
so is $\rho_{ \mathbb{H}^2}(a,b)\,. $
Therefore it will be often convenient below to consider a normalized point pair 
$a,b \in S^1(0,1)\cap  \mathbb{H}^2.$ 

\end{nonsec}

\begin{nonsec}{\bf Horizontal and vertical point pairs and $v_{\mathbb{H}^2}.$}\label{2.19}
We consider now two special locations for a point pair $\{a,b\}:$ (1) ${\rm Re}(a)=  
{\rm Re}(b)=0,$ 
(2) ${\rm Im}(a)=  {\rm Im}(b).$ In the case (1), the hyperbolic disk
 $B_{\rho_{\mathbb{H}^2}}(m,M)$, 
centered at the
hyperbolic midpoint $m$ of $a$ and $b$ and with the radius 
$M= \rho_{\mathbb{H}^2}(a,b)/2,$
  is clearly the same as the Euclidean disk 
 $B^2(i|m| {\rm ch} M,|m| {\rm sh} M), $ see \cite[Fig. 4.4 p.53]{hkv}.
 For brevity, write $ec=  i |m| {\rm ch} M, \, ec_1= ec+ |m|.$ 
 The circle $S^1(ec_1,|m| {\rm ch} M)$ passes through the points $a,b, |m|$ and is tangent
 to the real axis at the point $|m|.$ 
Let 
$$\{ a_1, b_1 \}= S^1(0,|m|)\cap S^1(i|m| {\rm ch} M,|m| {\rm sh} M). $$
Then $ \rho_{\mathbb{H}^2}(a_1,b_1) =2M.$ 
Because the inscribed angle $\measuredangle(a,|m|,b)$
of the circle $S^1(ec_1,|m| {\rm ch} M)$ equals half of the central angle 
$\measuredangle(a,ec_1,b)$, we see that
    $$\theta=v_{\mathbb{H}^2}(a,b)=
\measuredangle(a,|m|,b)= \measuredangle(ec,ec_1,b)\,, $$
 $$\sin \theta = \frac{|ec-b|}{|ec_1-b|}=\frac{{\rm sh} M}{{\rm ch} M} = 
 {\rm th}(\rho_{\mathbb{H}^2}(a,b)/2)$$
and also  

\begin{equation} \label{1stsinth}
\begin{cases}
{\displaystyle \theta=v_{\mathbb{H}^2}(a,b)=
\measuredangle(a,|m|,b)= \measuredangle(ec,ec_1,b)\,,}&\\[3mm]
  {\displaystyle \sin(\theta)={\rm th} \frac{\rho_{\mathbb{H}^2}(a,b)}{2}\,,}&\\[3mm]
 {\displaystyle \rho_{\mathbb{H}^2}(a,b)= \rho_{\mathbb{H}^2}(a_1,b_1)\,.}&
\end{cases}
\end{equation} 

These observations also show that $v_{\mathbb{H}^2}$ is not invariant under
M\"obius automorphisms of ${\mathbb{H}^2}$ whereas $\rho_{\mathbb{H}^2}$
has this invariance property. We also see that for $t>1$
\begin{equation}
v_{\mathbb{H}^2}(i,it) = {\rm arcsin}\left(\frac{t-1}{t+1}\right).
\end{equation}

In the case (2), trivially $v_{\mathbb{H}^2}(a,b)=
\measuredangle(a,{\rm Re}(a+b)/2,b)$.
\end{nonsec}
 
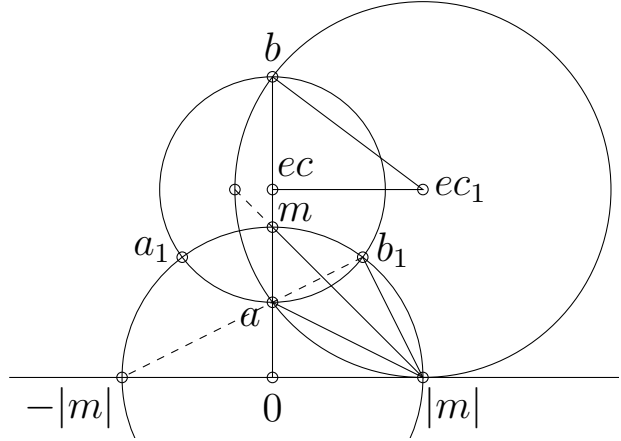
\begin{figure}
    \centering
    \begin{tikzpicture}
    \clip (-3.5,-0.8) rectangle (4.6,5);
    \draw (0,1) circle (0.7mm);
    \draw (0,4) circle (0.7mm);
    \draw (0,2) circle (0.7mm);
    \draw (0,2.5) circle (0.7mm);
    \draw (2,2.5) circle (0.7mm);
    \draw (-1.2,1.6) circle (0.7mm);
    \draw (1.2,1.6) circle (0.7mm);
    \draw (-0.5,2.5) circle (0.7mm);
    \draw (-2,0) circle (0.7mm);
    \draw (2,0) circle (0.7mm);
    \draw (0,0) circle (0.7mm);
    \draw (0,4) -- (0,0);
    \draw (0,4) -- (2,2.5) -- (0,2.5);
    \draw[dashed] (-0.5,2.5) -- (0,2);
    \draw (0,2) -- (2,0);
    \draw[dashed] (1.2,1.6) -- (-2,0);
    \draw (0,1) -- (2,0) -- (1.2,1.6);
    \draw (-3.5,0) -- (5,0);
    \draw (0,2.5) circle (1.5cm);
    \draw (0,0) circle (2cm);
    \draw (2,2.5) circle (2.5cm);
    \node[scale=1.3] at (-0.3,0.8) {$a$};
    \node[scale=1.3] at (0,4.4) {$b$};
    \node[scale=1.3] at (0.3,2.2) {$m$};
    \node[scale=1.3] at (0.3,2.8) {$ec$};
    \node[scale=1.3] at (2.5,2.5) {$ec_1$};
    \node[scale=1.3] at (-1.6,1.7) {$a_1$};
    \node[scale=1.3] at (1.6,1.7) {$b_1$};
    \node[scale=1.3] at (-2.7,-0.4) {$-|m|$};
    \node[scale=1.3] at (2.4,-0.4) {$|m|$};
    \node[scale=1.3] at (0,-0.4) {$0$};
    \end{tikzpicture}
    \caption{The hyperbolic disk $B_{\rho_{\mathbb{H}^2}}(m,M)$ centered at the hyperbolic \newline midpoint $m$ of $a$ and $b$ and with the radius $M= \rho_{\mathbb{H}^2}(a,b)/2.$ Now,
 $  v_{\mathbb{H}^2}(a_1,b_1)=
\measuredangle(a_1,0,b_1)=2  v_{\mathbb{H}^2}(a,b) \,,\quad  {\rm and}\,\,  \quad
 \rho_{\mathbb{H}^2}(a_1,b_1)=\rho_{\mathbb{H}^2}(a,b)\,. $}
    \label{fig1}
\end{figure}

\begin{nonsec}{\bf M\"obius transformations.}\label{mob}
We recall a few formulas  for some basic transformations. 
An inversion in the circle
$S^1(q,r)=\{ z \in \mathbb{C}: |z-q|=r\}$ is defined as follows:
\begin{equation} \label{inv}
z \mapsto q+ \frac{r^2}{\overline{z}-\overline{q}}.
\end{equation}
A reflection of the point $z$ in the line through the two points $a$ and $b$ is defined by
\begin{equation} \label{refl}
w(z) =  \frac{a-b}{\overline{a}-\overline{b}} \overline{z} -  
\frac{a\overline{b}- \overline{a}b}{\overline{a}-\overline{b}}.
\end{equation}

For  $a \in \mathbb{B}^2 \setminus\{0\},$ the  M\"obius automorphism
of  $\mathbb{B}^2$ with the formula
\begin{equation}\label{Ta}
T_a(z) = \frac{z-a}{1-\overline{a}z}\,,
\end{equation}
has the property $T_a(a) = 0.$ Moreover, $\pm a/|a|$ are fixed points of $T_a$, and
$T_a$ maps the diameter $[-a/|a|, a/|a|]$ of the unit disk onto itself.
This automorphism has a factorization in terms of an inversion and an orthogonal map
\cite[p.21]{a}, \cite[p.40, Thm 5.5.1]{b}
\begin{equation}\label{Ta2}
T_a(z)  = (p_a\circ \sigma_a)(z)\,; \quad p_a(z)=-\frac{a}{\overline{a}} \overline{z},\quad
\sigma_a(z)=\frac{a(1-(\overline{z}/\overline{a}))}{1-a\overline{z}}\,.
\end{equation}
Here, $\sigma_a$ is the inversion in the circle $S^1(1/\overline{a}, \sqrt{|a|^{-2}-1}),$
 orthogonal to the unit circle. 
\end{nonsec}

\begin{nonsec}{{\bf Similarity automorphism of $\mathbb{H}^2.$} 
}

For $ a,b\in\mathbb{H}^2 $ with ${\rm Im}(a)\neq {\rm Im}
(b)$,
let
$$
   c=LIS[a,b,0,1]
    =\frac{a\overline{b}-\overline{a}b}{a-\overline{a}-b+\overline{b}} 
    =\frac{{\rm Im}(a\overline{b})}{{\rm Im}(a)-{\rm Im}(b)}
$$
and
$ r^2= |a-c||b-c| $.

Then, we have 
\begin{equation}\label{eq:bcac}
   b-c=b-\frac{a\overline{b}-\overline{a}b}{a-\overline{a}-b+\overline{b}} 
      =\frac{(b-\overline{b})(a-b)}{a-\overline{a}-b+\overline{b}} 
   \quad {\rm and} \quad
   a-c=\frac{(a-\overline{a})(a-b)}{a-\overline{a}-b+\overline{b}} .
\end{equation}

For $ g(z)=c+\dfrac{r^2}{\overline{z-c}} $ and 
$ m(z)=c+\dfrac{|b-c|^2}{\overline{z-c}} $, 
the composition $ m\circ g $ is given as follows.
\begin{align*}
  m\circ g(z) & = c+\frac{|b-c|^2}{\overline{c}+\frac{r^2}{z-c}-\overline{c}}
                = c+\frac{|b-c|^2}{r^2}(z-c)\\
              & = c+\frac{|b-c|^2}{|a-c||b-c|}(z-c)
                = c+\frac{|b-c|}{|a-c|}(z-c).
\end{align*}
From \eqref{eq:bcac} and $ a,b\in\mathbb{H}^2 $,
\begin{align*}
  m\circ g(z)  & = c+\frac{|b-\overline{b}|}{|a-\overline{a}|}(z-c)
                 =c+\frac{b-\overline{b}}{a-\overline{a}}(z-c) \\
               & = \frac{{\rm Im}(b)}{{\rm Im}(a)}z
                  +\frac{{\rm Im}(a\overline{b})}{{\rm Im}(a)-{\rm Im}(b)}
                   \frac{{\rm Im}(a)-{\rm Im}(b)}{{\rm Im}(a)}\\
              & = \frac{{\rm Im}(b)}{{\rm Im}(a)}z
                  +\frac{{\rm Im}(a\overline{b})}{{\rm Im}(a)}.
\end{align*}

For the map $ m\circ g $, we have
$$
   m\circ g(a)=\frac{{\rm Im}(b)}{{\rm Im}(a)}a
                  +\frac{{\rm Im}(a\overline{b})}{{\rm Im}(a)}
              =\frac{(b-\overline{b})a+a\overline{b}-\overline{a}b}
                    {a-\overline{a}}
              =\frac{b(a-\overline{a})}{a-\overline{a}}=b.
$$
\end{nonsec} 

\begin{lem} \label{vHquot}
(1) Suppose $ a,b\in S^1(0,1)\cap \mathbb{H}^2=P$ and $ h $ is a
M\"obius mapping of $ \mathbb{H}^2 $ onto itself and of $ P $ onto itself.
Then,
$$
    1/2 < v_{\mathbb{H}^2}(h(a), h(b))/v_{\mathbb{H}^2}(a,b) <2
$$
holds.

(2) Suppose that $L$ is a line and
$ a,b\in L\cap \mathbb{H}^2$ and $ h $ is a
M\"obius mapping of $ \mathbb{H}^2 $ onto itself and of $  L\cap \mathbb{H}^2 $ onto itself.
Then,
$$
    v_{\mathbb{H}^2}(h(a), h(b))=v_{\mathbb{H}^2}(a,b) 
$$
holds.
\end{lem}

\bigskip

\noindent

\begin{proof}
(1) By the assumptions, $h$ can be written as
\begin{equation}\label{eq:mobH}
   h(z)=\frac{z-\tau}{1-\tau z}, \quad -1<\tau<1.
\end{equation}
If $ a=b $, $  v_{\mathbb{H}^2}(a,b)=v_{\mathbb{H}^2}(h(a), h(b))=0 $.
For $ a\neq b $, without loss of generality, we may assume
$ 0<\arg b<\arg a<\pi $.

Let $d\in\mathbb{R}$ be the point where $v_{\mathbb{H}^2}(a,b)$ is attained, i.e.
$v_{\mathbb{H}^2}(a,b)=\arg\Big(\dfrac{a-d}{b-d}\Big)>0$.

On the other hand, let $\tilde{d}\in\mathbb{R}$ be the point where
$ v_{\mathbb{H}^2}\big(h(a),h(b)\big) $ is attained.
Then, $ \tilde{d} $ can be written as $ \tilde{d}=h(d') $
for some $ d' $ with $-1<d'<1$.
Moreover, $ h $ preserves the orientation of points on the unit circle,
so we have $ 0<\arg h(b)<\arg h(a)<\pi $ and
$$ v_{\mathbb{H}^2}\big(h(a),h(b)\big)
    =\arg\Big(\dfrac{h(a)-h(d')}{h(b)-h(d')}\Big)>0\,,$$
and further
\begin{align*}
    \arg\Big(\dfrac{h(a)-h(d')}{h(b)-h(d')}\Big)
      &=\arg\Big(\dfrac{a-d'}{b-d'}\dfrac{b-\frac1{\tau}}{a-\frac1{\tau}}\Big)
       =\Big|\arg\Big(\dfrac{a-d'}{b-d'}\Big)
        +\arg\Big(\dfrac{b-\frac1{\tau}}{a-\frac1{\tau}}\Big)\Big|\\
     & \leq \Big|\arg\Big(\dfrac{a-d'}{b-d'}\Big)\Big|
        +\Big|\arg\Big(\dfrac{b-\frac1{\tau}}{a-\frac1{\tau}}\Big)\Big|.
\end{align*}
Since the first term on the right side represents
the angle $ \measuredangle(a,h(d'),b) $ with $ -1<h(d')<1 $,
$$
    \arg\Big(\dfrac{a-d'}{b-d'}\Big)<\arg\Big(\dfrac{a-d}{b-d}\Big)
$$
holds.
The second term represents the angle
$ \measuredangle(b,1/\tau,a) $ with $ |1/\tau|>1 $,
which is obviously smaller than the angle
$ \measuredangle(a,\tau',b) $ for any $ -1<\tau'<1 $.
So,
$$
     \Big|\arg\Big(\dfrac{b-\frac1{\tau}}{a-\frac1{\tau}}\Big)\Big|
       <\arg\Big(\dfrac{a-d}{b-d}\Big).
$$
Hence, $ v_{\mathbb{H}^2}\big(h(a),h(b)\big)<2v_{\mathbb{H}^2}(a,b) $ is obtained.
The other inequality is obtained by considering the inverse transformation.

(2) In this case $h$ must be a similarity transformation
and the claim is obviously true.
\end{proof}

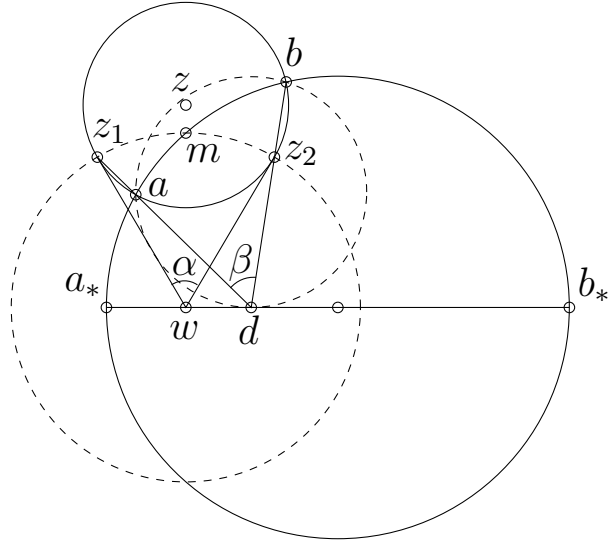
\begin{figure}
    \centering
    \begin{tikzpicture}
    \draw (0,1.5) circle (0.7mm);
    \draw (2,3) circle (0.7mm);
    \draw (0.666,2.321) circle (0.7mm);
    \draw (0.666,2.693) circle (0.7mm);
    \draw (-0.511,2) circle (0.7mm);
    \draw (1.843,2) circle (0.7mm);
    \draw (-0.390,0) -- (5.765,0);
    \draw (-0.390,0) circle (0.7mm);
    \draw (0.666,0) circle (0.7mm);
    \draw (1.535,0) circle (0.7mm);
    \draw (2.687,0) circle (0.7mm);
    \draw (5.765,0) circle (0.7mm);
    \draw (-0.511,2) -- (1.535,0) -- (2,3);
    \draw (-0.511,2) -- (0.666,0) -- (1.843,2);
    \draw[dashed] (1.535,1.535) circle (1.535cm);
    \draw (2.687,0) circle (3.077cm);
    \draw[dashed] (0.666,0) circle (2.321cm);
    \draw (0.666,2.693) circle (1.366cm);
    \node[scale=1.3] at (0.3,1.6) {$a$};
    \node[scale=1.3] at (2.1,3.4) {$b$};
    \node[scale=1.3] at (0.9,2.05) {$m$};
    \node[scale=1.3] at (-0.7,0.3) {$a_*$};
    \node[scale=1.3] at (1.5,-0.3) {$d$};
    \node[scale=1.3] at (6.1,0.3) {$b_*$};
    \node[scale=1.3] at (0.65,-0.3) {$w$};
    \node[scale=1.3] at (0.6,3) {$Z$};
    \node[scale=1.3] at (-0.35,2.35) {$z_1$};
    \node[scale=1.3] at (2.2,2.05) {$z_2$};
    \draw (0.82,0.3) arc (55:125:0.3);
    \node[scale=1.3] at (0.65,0.55) {$\alpha$};
    \draw (1.6,0.4) arc (85:140:0.4);
    \node[scale=1.3] at (1.4,0.6) {$\beta$};
    \end{tikzpicture}
    \caption{The points $a,b$ on the unit circle $S^1$ (on solid line), their hyperbolic middle point $m$, the hyperbolic circle $S_\rho(m,\rho_{\mathbb{H}^2}(a,b)/2)$ (on solid line), and its Euclidean center point $Z$. The points $z_1,z_2$ are the intersection points of the hyperbolic circle $S_\rho(m,\rho_{\mathbb{H}^2}(a,b)/2)$ and the Euclidean circle $S^1({\rm Re}(m),{\rm Im}(m))$ (on dashed line). The supremum $\sup\{\measuredangle(a,x,b)\text{ : }x\in\mathbb{R}^1\}$ is attained  at the point
     $d.$}
    \label{fig2}
\end{figure}

\begin{nonsec}{\bf General point pairs.} The above discussion \ref{2.19} of vertical 
and horizontal point pairs also
yields inequalities for the visual angle metric of general
point pairs $a,b \in \mathbb{H}^2,$  with $ {\rm Re}(a-b)\neq 0$  
and ${\rm Im}(a-b)\neq 0.$

Let $B^2(Z,r)= B_{\rho}(m,M)$ be a disk in  $\mathbb{H}^2$
and $w \in \mathbb{R}=\partial \mathbb{H}^2.$ The two
points $z_1,z_2$ in the intersection
 $S^1(Z,r)\cap S^1(w,|w-m|)$ of these
two orthogonal circles determine the tangent lines
$L[w,z_j], j=1,2,$ to $S^1(Z,r)$ at the points  $z_1,z_2.$ Now the angle
$\measuredangle(z_1,w,z_2)$ defines the visual angle of
$B^2(Z,r)$ at $w.$ It is clear that this angle is maximal
when $|w-m|$ is minimal, i.e. ${\rm Re}(w)={\rm Re}(m).$
We utilize these ideas in the proof of the next lemma.

\end{nonsec}

\bigskip

\noindent

\begin{lem} \label{genpairs}
 For $a,b \in  \mathbb{H}^2$, we have
\begin{equation} \label{eq:gene0}
{\rm arcsin} \left({\rm th}\frac{\rho_{ \mathbb{H}^2}(a,b)}{2} \right) \le 
v_{ \mathbb{H}^2}(a,b) \,,
\end{equation}
\begin{equation}\label{eq:gene}
  \frac{v_{ \mathbb{H}^2}(a,b)}{2} \le
  \frac{v_{ \mathbb{H}^2}(z_1,z_2)}{2} = {\rm arcsin} \left( {\rm th} \frac{\rho_{\mathbb{H}^2}(a,b)}{2}\right).
\end{equation}
Here $z_1$ and $z_2$ are as in Figure \ref{fig2}.
\end{lem}
\bigskip

\noindent
\begin{proof}
For the proof of \eqref{eq:gene0}, let $h$ be a M\"obius automorphism of $ \mathbb{H}^2$  which maps the semicircle
$S^1(w,r) \cap \mathbb{H}^2, w \in \mathbb{R},$ through $a$ and $b$ onto itself with
${\rm Im}(h(a))={\rm Im}(h(b))\,. $ Then, by the property of horizontal point pairs \ref{2.19},
$$
\frac{v_{ \mathbb{H}^2}(h(a),h(b))}{2}={\rm arcsin} \left({\rm th}
\frac{\rho_{ \mathbb{H}^2}(h(a),h(b))}{2} \right)= {\rm arcsin} \left({\rm th}
\frac{\rho_{ \mathbb{H}^2}(a,b)}{2} \right),
$$
which together with  Lemma \ref{vHquot}(1) yields
$$v_{\mathbb{H}^2}(a,b) \ge
\frac{v_{ \mathbb{H}^2}(h(a),h(b))}{2}
={\rm arcsin} \left({\rm th} \frac{\rho_{ \mathbb{H}^2}(a,b)}{2}\right)\,.
$$

For the proof of  \eqref{eq:gene}, fix  two points $ z_1 $ and $ z_2 $ with $$ {\rm Im}(z_1) ={\rm Im}(z_2) \,,
\quad  \sin\Big(\dfrac{v_{ \mathbb{H}^2}(a,b)}{2}\Big) =
   \sin \Big(\dfrac{v_{ \mathbb{H}^2}(z_1,z_2)}{2}\Big)\,,$$
which gives the maximum of $  \sin\Big(\dfrac{v_{ \mathbb{H}^2}(a,b)}{2}\Big) $.
Two points $ a $ and $ b $ with 
$$ {\rm Re}(a) ={\rm Re}(b) \,, \quad
  \sin\Big(\dfrac{v_{ \mathbb{H}^2}(a,b)}{2}\Big) =
   \dfrac12 \sin \Big(\dfrac{v_{ \mathbb{H}^2}(z_1,z_2)}{2}\Big)$$
give the minimum of $  \sin\Big(\dfrac{v_{ \mathbb{H}^2}(a,b)}{2}\Big) $.

For the general case, suppose ${\rm Re}(a)<{\rm Re}(b)$. Now, $v_{\mathbb{H}^2}(z_1,z_2)$ is the angle between the tangent lines at $z_1$ and $z_2$ on $S^1(Z,r)=\partial B_\rho(m,M)$ where $z_1,z_2$ and $Z$ are as  in Figure \ref{fig2}. Also, $v_{\mathbb{H}^2}(a,b)$ is the angle between $L[z_1,a]$ and $L[z_2,b]$. So, clearly,
\begin{align*}
v_{\mathbb{H}^2}(a,b)\leq v_{\mathbb{H}^2}(z_1,z_2).    
\end{align*}
Since $0\leq v_{\mathbb{H}^2}(z_1,z_2)\leq\pi$, we have
\begin{align*}
0\leq\frac{v_{\mathbb{H}^2}(a,b)}{2}\leq\frac{v_{\mathbb{H}^2}(z_1,z_2)}{2}\leq\frac{\pi}{2}.    
\end{align*}
Therefore, we have
\begin{align*}
\sin\left(\frac{v_{\mathbb{H}^2}(a,b)}{2}\right)\leq\sin\left(\frac{v_{\mathbb{H}^2}(z_1,z_2)}{2}\right).    
\end{align*}
So, the first inequality in \eqref{eq:gene} holds.

\medskip

Next, consider the latter equality in equation \eqref{eq:gene}.

Let $ S_{ab} $ be a circle passing through two points $a,b$
and orthogonal to the real axis.
Then $ S^1(Z,r) $ is a circle passing through points $ a,b $ and orthogonal
to $ S_{ab} $.
The center of $ S_{ab} $  is given by
\begin{equation}\label{eq:centerab}
    C_{ab}=m(a,\overline{a},b)
         =\frac{a\overline{a}-b\overline{b}}{a-b+\overline{a}-\overline{b}}
         =\frac{|a|^2-|b|^2}{2{\rm Re}(a-b)} \ (\in\mathbb{R}).
\end{equation}
The equations of the lines through $ C_{ab} $ and $ a $
and through $C_{ab}$ and $b$ are given by
\begin{align*}
   & (\overline{a}-C_{ab})z+(a-C_{ab})\overline{z}
         -(\overline{a}-C_{ab})a-(a-C_{ab})\overline{a}=0,\\
   & (\overline{b}-C_{ab})z+(b-C_{ab})\overline{z}
         -(\overline{b}-C_{ab})b-(b-C_{ab})\overline{b}=0.
\end{align*}
By eliminating $\overline{z}$ from the above two equations,
substituting \eqref{eq:centerab} into it,
and then removing unnecessary factors, we have
$$
    (a+b-\overline{a}-\overline{b})z-(2ab-a\overline{a}-b\overline{b})=0.
$$
(Here, we used the command ``resultant`` of Risa/Asir for eliminating
the variable $ \overline{z} $.)
Therefore, we have $ Z=(2ab-|a|^2-|b|^2)/(a+b-\overline{a}-\overline{b}) $
and
$$
   {\rm Im} Z=\frac{1}{2i}(Z-\overline{Z})
          =\frac{1}{2i}
         \frac{2(ab+\overline{a}\overline{b}-a\overline{a}-b\overline{b})}
              {a+b-\overline{a}-\overline{b}}
          =\frac{|a-\overline{b}|^2}{a+b-\overline{a}-\overline{b}}.
$$
The radius $ r $ of $ S^1(Z,r)$ is
$$
    r^2=(Z-a)(\overline{Z}-\overline{a})
       =-\frac{(a-b)(a-\overline{b})}{a+b-\overline{a}-\overline{b}}
         \frac{(\overline{a}-\overline{b})(\overline{a}-b)}
          {a+b-\overline{a}-\overline{b}}
       =\frac{|a-b|^2|a-\overline{b}|^2}{|a+b-\overline{a}-\overline{b}|^2}.
$$
So, we have
\begin{equation*}\label{eq:sinz}
   \sin \Big(\frac{v_{ \mathbb{H}^2}(z_1,z_2)}{2}\Big) 
     =\frac{r}{{\rm Im}(Z)}
     =\frac{|a-b||a-\overline{b}|}{|a+b-\overline{a}-\overline{b}|}
        \frac{|a+b-\overline{a}-\overline{b}|}{|a-\overline{b}|^2}
     =\frac{|a-b|}{|a-\overline{b}|}.
\end{equation*}
From \eqref{rhoH}, it follows that
$$
   {\rm th}\frac{\rho_{\mathbb{H}^2}(a,b)}{2}=\frac{|a-b|}{|a-\overline{b}|}.
$$
Hence, we have
$$
   \sin\Big(\frac{v_{ \mathbb{H}^2}(a,b)}{2}\Big) =   \sin \Big(\frac{v_{ \mathbb{H}^2}(z_1,z_2)}{2}\Big)
      ={\rm th}\frac{\rho_{\mathbb{H}^2}(a,b)}{2}=\frac{|a-b|}{|a-\overline{b}|}
$$
and the assertion follows because $v_{ \mathbb{H}^2}(a,b)<\pi.$ 
\end{proof}

\begin{cor} \label{genbds}
For $ a,b \in   {\mathbb{H}^2}$, we have
\begin{equation} \label{klvw319}
{\rm arctan}\left({\rm sh}\frac{\rho_{\mathbb{H}^2}(a,b)}{2} \right)
\le v_{\mathbb{H}^2}(a,b) \le 2\, {\rm arctan}\left({\rm sh}\frac{\rho_{\mathbb{H}^2}(a,b)}{2} \right).
\end{equation}
The lower bound is sharp for ${\rm Re}(a)={\rm Re}(b)$
and the upper bound for ${\rm Im}(a)={\rm Im}(b)\,.$
\end{cor}

\begin{proof} Observe first that 
 the following elementary relation holds for $u>0$:
\begin{equation*} \label{thsh}
{\rm arctan}({\rm sh}\,u)={\rm arcsin}({\rm th}\,u) \,.
\end{equation*}
This relation, together with Lemma \ref{genpairs}, yields
the inequality. The sharpness follows from \ref{2.19}.
\end{proof}

\begin{nonsec}{\bf Remark.}
The inequality \eqref{klvw319} was also given in
 \cite[Theorem 3.19]{klvw}. Our proof is, however,
 different from the earlier proof.
\end{nonsec}

\begin{prop} \label{uform20231209}
 For $ a,b, c\in S^1(0,1)  \cap  {\mathbb{H}^2}$ with 
 $0\le {\rm arg}\, c < \min \{{\rm arg}\, a\,, {\rm arg}\, b\}$,
 \[
\frac{|c+ab|}{|a+b|} < 1\,. 
 \]
\end{prop}

\begin{proof} 
For $ 0<\gamma<\beta<\alpha<\pi $,
let $ a=e^{i\alpha}, b=e^{i\beta}, c=e^{i\gamma} $.
Set $ \theta=\alpha-\beta $.

Then, $ |a+b|=2\cos\frac{\theta}2 $ holds.

\begin{enumerate}
  \item[{\bf Case 1.}] $ \alpha+\beta-\gamma <\pi $.
  
      Since $ \theta=\alpha-\beta<\alpha-\gamma<\alpha+\beta-\gamma\leq \pi $
      hold and the function $ f(t)=\cos t $ is monotonically decreasing in
      $ 0\leq t\leq \pi $,
      we have 
      $$
          |a+b|=2\cos \frac{\theta}2\geq 2\cos \frac{\alpha+\beta-\gamma}{2}
               = |c+ab|.
      $$
   
  \item[{\bf Case 2.}] $ \alpha+\beta-\gamma >\pi $.

      In this case, we have
      $ 0<2\pi-(\alpha+\beta-\gamma)<\pi $. 
      Therefore, the following holds
      $$
         |c+ab|=2\cos \frac{2\pi-(\alpha+\beta-\gamma)}{2}.
      $$
      Since $ 0<\gamma<\alpha<\pi $, it follows that
      $$
        2\pi> 2\alpha-\gamma =\alpha+\alpha-\gamma=\theta+\beta+\alpha-\gamma.
      $$
      Therefore, we have
      $$
         (\pi>)\ 2\pi-(\alpha+\beta-\gamma)>\theta \ (>0).
      $$
      Hence, 
      $$
          |c+ab|=2\cos\frac{2\pi-(\alpha+\beta-\gamma)}{2}
             <2\cos\frac{\theta}{2}=|a+b|
      $$
      holds.
\end{enumerate}
     In conclusion, $ |c+ab|<|a+b| $ holds in both cases.
\end{proof}


  
 \section{Formulas for the points of intersection}

For the proof of the main results for given points $a,b \in \mathbb{H}^2$, we need 
explicit formulas for the points of intersection of hyperbolic lines. 
These  formulas will be given in the course of this work, starting in this section.


\begin{lem}\label{uformula} 
(1) For $ a,b\in  {\mathbb{H}^2} $,
$$ 
LIS[a,\overline{b},b,\overline{a}]= \frac{{\rm Im}(a b)}{{\rm Im}(a+b)}\,.
$$

(2) For $ a,b\in S^1(0,1) \cap {\mathbb{H}^2} $
with $ {\rm Re}\, (a-b)\neq0 $ and $ a b\neq 1 $, we have
\[
 LIS[a,\overline{b},b,\overline{a}]= \frac{1+a b}{a+b} \in (-1,1)\,.
\]

(3)  For $ a,b\in  {\mathbb{H}^2} $ with ${\rm Re}(a-b) \neq 0$,
\[
u_1\equiv LIS[0, 1, (a + b)/2, (a + b)/2 + i(b - a)]= 
\frac{|a|^2-|b|^2}{2\,{{\rm Re}(a-b)}}
\]
and $u_1=0$ for  $ a,b\in S^1(0,1) \cap {\mathbb{H}^2} .$ 
\end{lem}

\begin{proof} (1) From \eqref{LIS}, we obtain
\begin{align*}
u &\equiv LIS[a,\overline{b},b,\overline{a}]\\
  &= \frac{(\overline{a} \overline{b}-a b)(b-\overline{a})
           -(a-\overline{b})(\overline{b} \overline{a} -b a)}
          {(\overline{a}-b)(b-\overline{a})-(a-\overline{b})(\overline{b}-a)}
   = \frac{(\overline{a} \overline{b}-a b)(b-\overline{a}-a+\overline{b})}
          {-(\overline{a}-b)^2+(a-\overline{b})^2}\\
  &= \frac{(\overline{a} \overline{b}-a b)(-a+\overline{b}-\overline{a}+b)}
          {(a-\overline{b}-\overline{a}+b)(a-\overline{b}+\overline{a}-b)}
   = \frac{a b-\overline{a} \overline{b}}{a+b-(\overline{a}+\overline{b})}=\frac{{\rm Im}(a b)}{{\rm Im}(a+b)}.
\end{align*}
(2) From $ a,b\in S^1(0,1) $, it follows that $ a \overline{a}=1 $ and 
$ b \overline{b}=1 $ . 
Therefore we have
\begin{align*}
 u &= \frac{a b-\dfrac{1}{a b}}{a+b-\Big(\dfrac1a+\dfrac1b\Big)}
    = \frac{(a b)^2-1}{a b(a+b)-(a+b)}
    = \frac{(a b-1)(a b+1)}{(a b-1)(a+b)}
    = \frac{a b+1}{a+b}.
\end{align*}
\bigskip
\noindent
By setting $c=1$ in Proposition \ref{uform20231209}, we see that $|u|<1.$

(3) Follows from \eqref{eq:aster2} and Lemma \ref{epF} because 
$u_1=(a_{\ast}+ b_{\ast})/2.$

\end{proof}

\bigskip


\begin{figure}
    \centering
    \begin{tikzpicture}[scale=0.7]
    \clip (-6,-1) rectangle (6.5,11.3);
    \draw (0,1.5) circle (1mm);
    \draw (2,3) circle (1mm);
    \draw (0.666,2.321) circle (1mm);
    \draw (2,3) -- (-2,0);
    \draw (-5.535,10.964) circle (1mm);
    \draw (-5.535,10.964) circle (10.964cm);
    \draw[dashed] (-5.535,10.964) -- (-5.535,0);
    \draw (0,0.75) parabola (-6,12.750);
    \draw (0,0.75) parabola (6.990,17.036);
    \draw (-5.535,10.964) -- (2.687,0);
    \draw (1.535,1.535) circle (1mm);
    \draw[dashed] (1.535,1.535) -- (1.535,0);
    \draw (2,1.5) parabola (-6,12.166);
    \draw (2,1.5) parabola (6.990,5.650);
    \draw (-6,0) -- (7,0);
    \draw (-5.535,0) circle (1mm);
    \draw (-2,0) circle (1mm);
    \draw (-0.390,0) circle (1mm);
    \draw (1.535,0) circle (1mm);
    \draw (2.687,0) circle (1mm);
    \draw (5.765,0) circle (1mm);
    \draw (-2,0) circle (3.535cm);
    \draw (1.535,1.535) circle (1.535cm);
    \draw (2.687,0) circle (3.077cm);
    \node[scale=1.3] at (-0.4,1.6) {$a$};
    \node[scale=1.3] at (2,3.5) {$b$};
    \draw[fill=white,color=white] (-0.1,2.90) -- (0.4,2.90) -- (0.4,2.45) -- (-0.1,2.45) -- (-0.1,2.90);
    \node[scale=1.3] at (0.2,2.65) {$m$};
    \node[scale=1.3] at (-5.1,11) {$q$};
    \node[scale=1.3] at (1.55,1.9) {$p$};
    \node[scale=1.3] at (-2,-0.4) {$c$};
    \node[scale=1.3] at (1.75,-0.4) {$d$};
    \node[scale=1.3] at (0.14,-0.4) {$a_*$};
    \node[scale=1.3] at (6.2,-0.5) {$b_*$};
    \node[scale=1.3] at (-4.3,-0.45) {$2c-d$};
    \end{tikzpicture}
    \caption{The points $p$ and $q$ are the centers of the two circles through $a$ and $b,$ tangent to the real axis at $d$ and $2c-d,$ resp. The points $p$ and $q$ are also the points of intersection of the two parabola with the real axis as the directrix and with foci $a$ and $b,$ resp. Because ${\rm Im}(p)<{\rm Im}(q)$, we see that
    $v_{\mathbb{H}^2}(a,b)=
    \max\{ \measuredangle(a,d,b),  \measuredangle(a,2c-d,b)\}= \measuredangle(a,d,b)\,.$
    The point $m$ here is the hyperbolic midpoint of $a$ and $b.$}
    \label{fig3}
\end{figure}
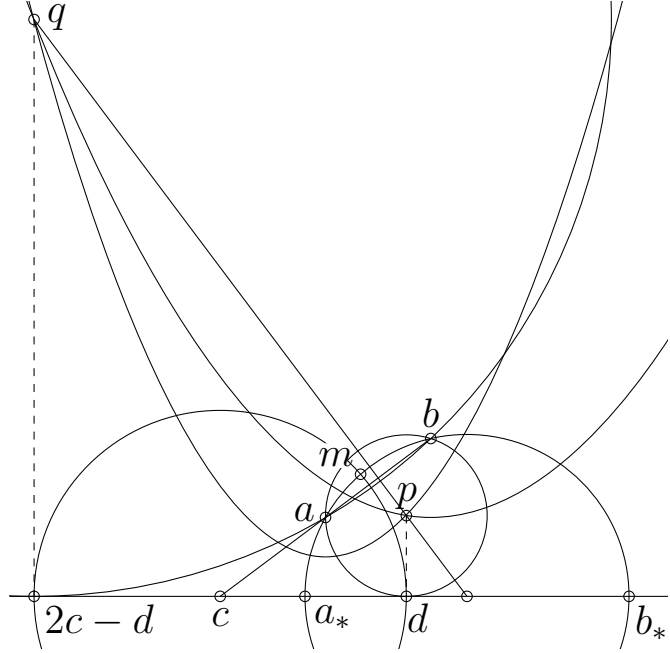

For $a,b\in\mathbb{H}^2 $ with  ${\rm Im}(a-b) \neq 0\, $, there are 
two circles through $a$ and $b$, tangent to the real axis. Let $d$ 
be the point of the smaller circle on the real axis. 

\bigskip

\begin{lem}\label{dForm} 
Let $ a,b\in\mathbb{H}^2 $ with ${\rm Im}(a-b) \neq 0\,. $ Then the above point
$d$ is given by
$$
    d=\frac{{\rm Im}\,( a\overline{b})
            +{\rm sgn}\big({\rm Re}\,(a-b)\big)|a-b|
            \sqrt{{{\rm Im}}\,(a){{\rm Im}}\,(b)}}
           {{\rm Im}\,(a-b)}.
$$
\end{lem}
\bigskip

\noindent
\begin{proof}

Let $c_0$ be the center of the circle $C(a,b,d)$, that is
$$
  c_0=m(a,b,d)
        =\frac{a\overline{a}(b-d)+b\overline{b}(d-a)+d^2(a-b)}
           {a(d-\overline{b})+b(\overline{a}-d)
            +d(\overline{b}-\overline{a})}
$$
by \eqref{mfun}. The points $z$ of $ C(a,b,d) $ satisfy the following equation
$$
   |z-c_0|=|d-c_0|\,.
$$
This equation is equivalent to
\begin{align*}
& \big(a(d-\overline{b})+b(\overline{a}-d)
   +d(\overline{b}-\overline{a})\big)z\overline{z}
  +\big(a\overline{a}(\overline{b}-d)+b\overline{b}(d-\overline{a})
      +d^2(\overline{a}-\overline{b})\big)z\\
& \quad
  -\big(a\overline{a}(b-d)+b\overline{b}(d-a)+d^2(a-b)\big)\overline{z}
  -d\big((\overline{a}b-\overline{b}a)d+ab(\overline{b}-\overline{a})
     +\overline{a}\overline{b}(a-b)\big)=0.
\end{align*}
Since this equation has the double root $ z=d $,
the discriminant of the following quadratic equation is zero.
That is,
\begin{align*}
& \big(a(d-\overline{b})+b(\overline{a}-d)
   +d(\overline{b}-\overline{a})\big)z^2
  +\big(a\overline{a}(\overline{b}-b)+b\overline{b}(a-\overline{a})
      +d^2(\overline{a}-\overline{b}-a+b)\big)z\\
& \quad
  -d\big((\overline{a}b-\overline{b}a)d+ab(\overline{b}-\overline{a})
     +\overline{a}\overline{b}(a-b)\big)=0
\end{align*}
and
\begin{equation}\label{eq:discrD}
 Discr=\big((a-b-\overline{a}+\overline{b})d^2
      -2(a\overline{b}-\overline{a}b)d
        -ab(\overline{a}-\overline{b})+\overline{a}\overline{b}(a-b)\big)^2=0.
\end{equation}
Equation \eqref{eq:discrD} has two solutions
$$
  d^+=\frac{{\rm Im}\,(a\overline{b})+|a-b|\sqrt{{\rm Im}\,(a){\rm Im}\,(b)}}
           {{\rm Im}\,(a-b)},\quad
  d^-=\frac{{\rm Im}\,(a\overline{b})-|a-b|\sqrt{{\rm Im}\,(a){\rm Im}\,(b)}}
           {{\rm Im}\,(a-b)}.
$$
To find out which one of these two solutions gives $ d $, consider 
the following four cases:
\begin{itemize}
  \item  The case 
     $ {\rm Im}\,(a)>{\rm Im}\,(b) $ and $ {\rm Re}\,(a)>{\rm Re}\,(b)$,
     then $ d=\max\{d^+,d^-\}=d^+$.

  \item  The case 
     $ {\rm Im}\,(a)>{\rm Im}\,(b) $ and $ {\rm Re}\,(a)<{\rm Re}\,(b)$,
     then $ d=\min\{d^+,d^-\}=d^-$.

  \item  The case 
     $ {\rm Im}\,(a)<{\rm Im}\,(b) $ and $ {\rm Re}\,(a)>{\rm Re}\,(b)$,
     then $ d=\min\{d^+,d^-\}=d^+$.

  \item  The case 
     $ {\rm Im}\,(a)<{\rm Im}\,(b) $ and $ {\rm Re}\,(a)<{\rm Re}\,(b)$,
     then $ d=\max\{d^+,d^-\}=d^-$.
\end{itemize}
Hence, the sign before the square root of the solution 
is determined by the sign of the real part of $ a-b$,
and the assertion is obtained.
\end{proof}

\begin{nonsec}{\bf Remark.} 
Theorem 5.6(2) from \cite{frv} also gives the formula of Lemma \ref{dForm} but we prefer
the above more direct proof.

\begin{lem} \label{myFD2}  For $a,b\in S^1(0,1)\cap \mathbb{H}^2,$ the following 
formula  holds: 
$$ d=\frac{1}{1+ab}\Big(
        a+b-2 {\rm sgn}\big( {\rm Re}(a+b)\big)
            \sqrt{ab {\rm Im}(a)\cdot {\rm Im}(b)}
           \Big).$$
\end{lem}

\noindent
\begin{proof}
Assume that $ a,b\in S^1(0,1) $.  
By \eqref{hmid2}, $ d=u/(1+\sqrt{1-|u|^2}) $. By substituting $ u=\dfrac{1+ab}{a+b} $ as in Lemma \ref{uformula} 
into $ d=u/(1+\sqrt{1-|u|^2}) $, we have
\begin{align*}
 d&=\dfrac{u}{1+\sqrt{1-|u|^2}}=\dfrac{1-\sqrt{1-|u|^2}}{u}\\
  &=\dfrac{a+b}{1+ab}\bigg(1-\sqrt{1-\Big(\frac{1+ab}{a+b}\Big)^2}\bigg)
   =\dfrac{a+b}{1+ab}\bigg(1-\dfrac{1}{\sqrt{(a+b)^2}}
               \sqrt{-(a^2-1)(b^2-1)}\bigg)\\
  &=\dfrac{a+b}{1+ab}\bigg(1
         -\dfrac{1}{\sqrt{(a+b)^2}}
                \sqrt{-ab\Big(a-\frac1a\Big)\Big(b-\frac1b\Big)}\bigg)\\
  &=\dfrac{a+b}{1+ab}\bigg(1-\dfrac{2}{\sqrt{(a+b)^2}}
             \sqrt{ab\, {\rm Im}(a)\cdot {\rm Im}(b)}\bigg)\,.
\end{align*}
Here, set $ a+b=re^{i\theta} $.
Then, we have
$$
   \dfrac{a+b}{\sqrt{(a+b)^2}}
    =\begin{dcases}
      1  &  {\rm if} \ 0\leq\theta\leq \frac{\pi}2,\\
      -1 &  {\rm if} \ \theta> \frac{\pi}2 \,.
   \end{dcases}
$$
The above is equivalent to
$$
    {\rm sgn}\big( {\rm Re}(a+b)\big)
    =\begin{dcases}
      1  &  {\rm if} \ 0\leq\theta\leq \frac{\pi}2,\\
      -1 &  {\rm if} \ \theta> \frac{\pi}2\, .
   \end{dcases}
$$
\end{proof}
\bigskip

\begin{rem} \label{dishm}
In the case of $ a,b\in S^1(0,1) $, if $ a\neq b $,
the equation \eqref{eq:discrD} can be written as follows
$$
  (1+ab)d^2-2(a+b)d+(1+ab)=0.
$$ 
So, we have
$$
   d^2-2\dfrac{a+b}{1+ab}d+1 =d^2-\frac{2}{u}d+1  =0,
$$
and
$$
   d=\frac{1\pm\sqrt{1-u^2}}{u}.
$$
In conclusion, we see by \eqref{hmid2} that $d$ is the hyperbolic 
midpoint of $0$ and $u$ in the hyperbolic geometry of $\mathbb{B}^2.$ 
\end{rem}

\end{nonsec}

\begin{table}[ht]
    \centering
    \begin{tabular}{|l|l|l|}
       \hline
       Point & Definition & Formula\\
       \hline
       \multirow{3}{1em}{$a_*$} & \multirow{3}{18em}{$C[a,\overline{a},b]\cap\mathbb{R}^1$ (Points in order $a_*,a,b,b_*$ on the circle $C[a,\overline{a},b]$)} & \multirow{3}{18em}{$\dfrac{|a|^2-|b|^2+|a-b||a-\overline{b}|}{2 \,{\rm Re}(a-b)}$}\\
       &&\\
       &&\\
       \hline
       \multirow{3}{1em}{$b_*$}  & \multirow{3}{18em}{$C[a,\overline{a},b]\cap\mathbb{R}^1$ (Points in order $a_*,a,b,b_*$ on the circle $C[a,\overline{a},b]$)} & \multirow{3}{18em}{$\dfrac{|a|^2-|b|^2-|a-b||a-\overline{b}|}{2 \,{\rm Re}(a-b)}$}\\
       &&\\
       &&\\
       \hline
       \multirow{3}{1em}{$c$} & \multirow{3}{18em}{$LIS[a,b,0,1]$} & \multirow{3}{18em}{$\dfrac{{\rm Im}(a\overline{b})}{{\rm Im}(a-b)}$}\\
       &&\\
       &&\\
       \hline
       \multirow{3}{1em}{$d$}  & \multirow{3}{18em}{$d\in\mathbb{R}^1$ such that $\measuredangle(a,d,b)=\sup\{\measuredangle(a,z,b)\text{ : }z\in\mathbb{R}^1\}$} & \multirow{3}{19.5em}{$\dfrac{{\rm Im}(a\overline{b})+{\rm sgn}({\rm Re}(a-b))|a-b|\sqrt{{\rm Im}(a){\rm Im}(b)}}{{\rm Im}(a-b)}$}\\
       &&\\
       &&\\
       \hline
       \multirow{3}{1em}{$m$}  & \multirow{3}{18em}{Hyp. midpoint of $a,b$} & \multirow{3}{18em}{$\dfrac{{\rm Im}(ab)+i|a-\overline{b}|\sqrt{{\rm Im}(a){\rm Im}(b)}}{{\rm Im}(a+b)}$}\\
       &&\\
       &&\\
       \hline
       \multirow{3}{1em}{$p$} & \multirow{3}{18em}{$LIS[\dfrac{a+b}{2},\dfrac{a+b}{2}+i(b-a),d,d+i]$} & \multirow{3}{18em}{$\dfrac{|a|^2-|b|^2-2(a-b)d}{2{\rm Im}(a-b)}i$}\\
       &&\\
       &&\\
       \hline
       \multirow{3}{1em}{$q$} & \multirow{3}{18em}{$LIS[\dfrac{a+b}{2},\dfrac{a+b}{2}+i(b-a),2c-d,2c-d+i]$} & \multirow{3}{18em}{$\dfrac{|a|^2-|b|^2-2(a-b)f}{2{\rm Im}(a-b)}i$}\\
       &&\\
       &&\\
       \hline
       \multirow{3}{1em}{$u$}  & \multirow{3}{18em}{$LIS[a,\overline{b},b,\overline{a}]$} & \multirow{3}{18em}{$\dfrac{{\rm Im}(ab)}{{\rm Im}(a+b)}$}\\
       &&\\
       &&\\
       \hline
       \multirow{3}{1em}{$u_1$}  &  \multirow{3}{18em}{$LIS[0,1,\dfrac{a+b}{2},\dfrac{a+b}{2}+i(b-a)]$} & \multirow{3}{18em}{$\dfrac{|a|^2-|b|^2}{2\,{\rm Re}(a-b)}$}\\
       &&\\
       &&\\
       \hline
    \end{tabular}
    \vspace{2mm}
    \caption{Definitions and formulas of the points in Figures \ref{fig3} and \ref{fig4} for $a,b\in\mathbb{H}^2$ in the general case. Here, $f=\dfrac{{\rm Im}(a\overline{b})-{\rm sgn}({\rm Re}(a-b))|a-b|\sqrt{{\rm Im}(a){\rm Im}(b)}}{{\rm Im}(a-b)}$.}
    \label{t1}
\end{table}

\begin{table}[ht]
    \centering
    \begin{tabular}{|l|l|}
       \hline
       Point & Formula\\[1mm]
       \hline
       $a_*$ & ${\rm sgn}({\rm Re}(a-b))$\\[1mm]
       \hline
       $b_*$ & $-{\rm sgn}({\rm Re}(a-b))$\\[1mm]
       \hline
       \multirow{2}{1em}{$c$} & \multirow{2}{15em}{$\dfrac{a+b}{ab+1}$}\\
       &\\
       \hline
       \multirow{3}{1em}{$d$} & \multirow{3}{18em}{$\dfrac{a+b-2{\rm sgn}({\rm Re}(a+b))\sqrt{ab{\rm Im}(a)\cdot{\rm Im}(b)}}{1+ab}$}\\
       &\\
       &\\
       \hline
       \multirow{2}{1em}{$p$} & \multirow{2}{15em}{$\dfrac{2ab}{ab+1}d$}\\
       &\\
       \hline
       \multirow{2}{1em}{$q$} & \multirow{2}{15em}{$\dfrac{2ab}{ab+1}g$}\\
       &\\
       \hline
       \multirow{2}{1em}{$u$} & \multirow{2}{15em}{$\dfrac{1+ab}{a+b}$}\\
       &\\
       \hline
       $u_1$ & 0\\
       \hline
    \end{tabular}
    \vspace{2mm}
    \caption{Formulas of the points in Table \ref{t1} for $a,b\in\mathbb{H}^2$ when $|a|=|b|=1$, $ab\neq1$, $a\neq b$. Here, $g=\dfrac{a+b+2{\rm sgn}({\rm Re}(a+b))\sqrt{ab{\rm Im}(a)\cdot{\rm Im}(b)}}{1+ab}$.}
    \label{t2}
\end{table}  

\bigskip

\begin{figure}
    \centering
    \begin{tikzpicture}
    \draw (0,1.5) circle (0.7mm);
    \draw (2,3) circle (0.7mm);
    \draw (0.666,2.321) circle (0.7mm);
    \draw (0,-1.5) circle (0.7mm);
    \draw (2,-3) circle (0.7mm);
    \draw (0,1.5) -- (2,-3);
    \draw (2,3) -- (0,-1.5);
    \draw (-0.390,0) -- (5.765,0);
    \draw (-0.390,0) circle (0.7mm);
    \draw (0.666,0) circle (0.7mm);
    \draw (1.535,0) circle (0.7mm);
    \draw (2.687,0) circle (0.7mm);
    \draw (5.765,0) circle (0.7mm);
    \draw (1.535,1.535) circle (1.535cm);
    \draw (2.687,0) circle (3.077cm);
    \draw[dashed] (0.666,2.321) arc (41.0:-41.0:3.535);
    \draw (0,1.5) arc (60.8:-12.9:4.103);
    \draw (2,3) arc (12.9:-60.8:4.103);
    \node[scale=1.3] at (-0.3,1.7) {$a$};
    \node[scale=1.3] at (2,3.4) {$b$};
    \node[scale=1.3] at (1.1,2.35) {$m$};
    \node[scale=1.3] at (-0.7,0.3) {$a_*$};
    \node[scale=1.3] at (0.3,-0.2) {$u$};
    \node[scale=1.3] at (1.85,0.3) {$d$};
    \node[scale=1.3] at (5.5,0.4) {$b_*$};
    \node[scale=1.3] at (3,-0.21) {$u_1$};
    \end{tikzpicture}
    \caption{Here $u=LIS[a, \overline{b}, b, \overline{a}]$ and the dashed circular arc is a subarc of the circle centered at the point $c=LIS[a,b,0,1],$ orthogonal to the circle through $a,\overline{a},$ and $b.$  Moreover, $v_{\mathbb{H}^2}(a,b)= \measuredangle(a,d,b).$}
    \label{fig4}
\end{figure}
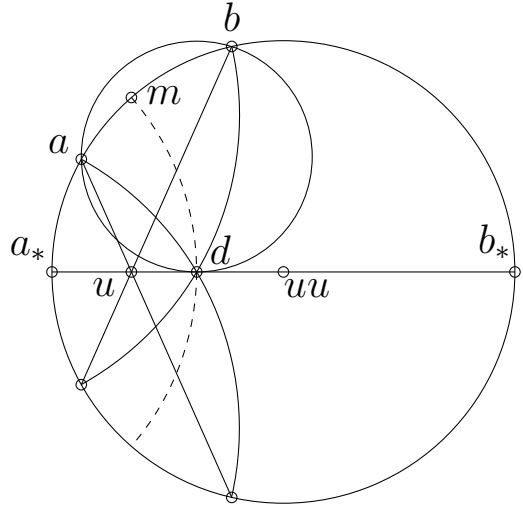
  
\begin{rem}\label{vHlow} Recall that, by Lemma \ref{uformula}(2), 
$u=(1+ a b)/(a+b)$ for $a,b \in \mathbb{H}^2 \cap S^1(0,1).$ 
 It is clear by Figure \ref{fig4} that
 $$v_{\mathbb{H}^2}(a,b)= \measuredangle(a,d,b) 
 \ge \measuredangle(a,u,b) =\left|{\rm arg}\left(\frac{a-u}{b-u}\right)\right| = 
 \left|{\rm arg}\left(\frac{a^2-1}{b^2-1}\right)\right|.$$ 
\end{rem}

\bigskip

\begin{lem} \label{hypMid24} 
(1) The hyperbolic midpoint $m$ of $ a,b\in\mathbb{H}^2$ is
$$
   m=\frac{{\rm Im}(ab)+i|a-\overline{b}|\sqrt{{\rm Im}(a){\rm Im}(b)}}
          {{\rm Im}(a+b)}.
$$
(2) $ LIS[a,\overline{b},\overline{a},b] = \frac{{\rm Im}(a b)}{{\rm Im}(a+b)}= {\rm Re}(m)\,.$
\end{lem}
\bigskip
   
\noindent
\begin{proof}
(1) By Proposition \ref{rhoabsr}  and Lemma \ref{quadr}, the point
$ m $  satisfies
$$
   \rho_{\mathbb{H}^2}(a,m)=\frac{|a-\overline{m}|+|a-m|}{|a-\overline{m}|-|a-m|}
   =\frac{|b-\overline{m}|+|b-m|}{|b-\overline{m}|-|b-m|}=\rho_{\mathbb{H}^2}(b,m).
$$
Therefore, we have
\begin{align*}
    (|a-\overline{m}|+|a-m|)(|b-\overline{m}|-|b-m|)
          &= (|a-\overline{m}|-|a-m|)(|b-\overline{m}|+|b-m|)\\
    |a-m||b-\overline{m}|
          &= |a-\overline{m}||b-m|,
\end{align*}
and
\begin{align*}
   & |a-m|^2|b-\overline{m}|^2- |a-\overline{m}|^2|b-m|^2\\
   & = (a-m)(\overline{a}-\overline{m})(b-\overline{m})(\overline{b}-m)
     - (a-\overline{m})(\overline{a}-{m})(b-{m})(\overline{b}-\overline{m})\\
   & = (m-\overline{m})
       \big((a-b-\overline{a}+\overline{b})m\overline{m}
        -(a\overline{b}-\overline{a}b)m
         -(a\overline{b}-\overline{a}b)\overline{m}
         +(a-b)\overline{a}\overline{b}-(\overline{a}-\overline{b})ab\big)=0.
\end{align*}
Hence, the point $ m $ is on the circle defined by the equation
\begin{equation}\label{eq:m-on-circle}
       (a-b-\overline{a}+\overline{b})z\overline{z}
        -(a\overline{b}-\overline{a}b)z
         -(a\overline{b}-\overline{a}b)\overline{z}
         +(a-b)\overline{a}\overline{b}-(\overline{a}-\overline{b})ab=0.
\end{equation}
The point $ m $ is also a point on the circle $C$ passing through 
the three points $ a,b,\overline{a},$  see \ref{geod1}.
The general form of the circle $ C $ is given by  
\begin{equation}\label{eq:m-on-C}
  (a-b+\overline{a}-\overline{b})z\overline{z}
   -(a\overline{a}-b\overline{b})z-(a\overline{a}-b\overline{b})\overline{z}
   +ab(\overline{a}-\overline{b})+\overline{a}\overline{b}(a-b)=0,
\end{equation}
and the point $ m $ satisfies this equation.
Eliminating $ \overline{z} $ from \eqref{eq:m-on-circle} and 
\eqref{eq:m-on-C}, we obtain
$$
   |a-b|^2\big((a+b-\overline{a}-\overline{b})z^2
    -2(ab-\overline{a}\overline{b})z+ab(\overline{a}+\overline{b})
     -\overline{a}\overline{b}(a+b)\big)=0.
$$
(Here, the command ``resultant`` of Risa/Asir was used to eliminate 
$ \overline{z} $.)
As $ m\in\mathbb{H}^2 $, we have
$$
  m=\frac{ab-\overline{a}\overline{b}
        -|a-\overline{b}|\sqrt{-(a-\overline{a})(b-\overline{b})}}
         {a+b-\overline{a}-\overline{b}}
   =\frac{{\rm Im}(ab)+i|a-\overline{b}|\sqrt{{\rm Im}(a){\rm Im}(b)}}
         {{\rm Im}(a+b)}.
$$

(2) 
\begin{align*}
  LIS[a,\overline{b},\overline{a},b]
                  &=\frac{(\overline{a} \overline{b}-a b)
                          (\overline{a}-b)-(a-\overline{b})
                          (a b -\overline{a} \overline{b})}
                        {(\overline{a}-b)(\overline{a}-b)
                             -(a-\overline{b})(a-\overline{b})}
                   = \frac{(\overline{a} \overline{b}-a b)
                             (\overline{a} -b+a-\overline{b})}
                         {(\overline{a}-b+a-\overline{b})
                            (\overline{a}-b-a+\overline{b})}\\
                 &=\frac{-2i \textrm{Im}(ab)}{-2i\textrm{Im}(a+b)}
                  =\frac{ \textrm{Im}(ab)}{\textrm{Im}(a+b)}
                 =\textrm{Re}(m).
\end{align*}

\end{proof}
  
\begin{lem}\label{lem_p}
For $ a,b\in\mathbb{H}^2$, ${\rm Re}(a-b)\neq 0$, and $ {\rm Im}(a-b)\neq0 $,
$$
  p=\frac{|a|^2-|b|^2-2(a-b)d}{2{\rm Im}(a-b)}i,
$$
where $d$ is as in Lemma \ref{dForm}. 
\end{lem}
\begin{proof}
The real part of the point $ p $ is $ d $, 
and it is a point on the perpendicular bisector of 
the line segment that connects the two points $ a $ and  $ b $.
The equation of the line is given by
\begin{equation}\label{eq:ab-orth}
  (\overline{a}-\overline{b})z+(a-b)\overline{z}
   -a\overline{a}+b\overline{b}=0.
\end{equation}
By setting $ p=d+yi $ and substituting $ p $ into \eqref{eq:ab-orth}, we obtain
$$
  (a-b-\overline{a}+\overline{b})yi-(a-b+\overline{a}-\overline{b})d
      +a\overline{a}-b\overline{b} =0.
$$
Therefore, we have
$$
   y=\frac{|a|^2-|b|^2-2{\rm Re}(a-b)d}{2{\rm Im}(a-b)},
$$
and
\begin{align*}
   p &= d+yi
      = d+\frac{|a|^2-|b|^2-2{\rm Re}(a-b)d}{2{\rm Im}(a-b)}i \\
     &= d\Big(1-\frac{{\rm Re}(a-b)}{{\rm Im}(a-b)}i\Big)+
        \frac{|a|^2-|b|^2}{2{\rm Im}(a-b)}i
      = -d\frac{a-b}{{\rm Im}(a-b)}i+\frac{|a|^2-|b|^2}{2{\rm Im}(a-b)}i\\
     &=\frac{|a|^2-|b|^2-2d(a-b)}{2{\rm Im}(a-b)}i.
\end{align*}
Hence, we have the assertion.    
\end{proof}

\begin{lem}
For $ a,b\in\mathbb{H}^2$, ${\rm Re}(a-b)\neq 0$, and $ {\rm Im}(a-b)\neq0 $,
$$
  q=\frac{|a|^2-|b|^2-2(a-b)f}{2{\rm Im}(a-b)}i,
$$
where
$$
  f=\frac{{\rm Im}(a\overline{b})-{\rm sgn}\big({\rm Re}(a-b)\big)|a-b|
           \sqrt{{\rm Im}(a){\rm Im}(b)}}{{\rm Im}(a-b)}.
$$    
\end{lem}
\begin{proof}
The proof is obtained by the same way as that of Lemma \ref{lem_p}.    
\end{proof}

\begin{cor}
For $ a,b\in S^1(0,1)\cap\mathbb{H}^2$, 
${\rm Re}(a-b)\neq 0$, and $ {\rm Im}(a-b)\neq0 $,
$$
  p=\dfrac{2ab}{1+ab}d,
$$
where $d$ is as in Lemma \ref{myFD2}.
\end{cor}
\begin{proof}
For $ a,b\in S^1(0,1)\cap\mathbb{H}^2 $,
the equation \eqref{eq:ab-orth} can be written as
\begin{equation}\label{eq:ab-orthS}
  (a-b)z-ab(a-b)\overline{z}=0.
\end{equation}
Setting $ p=d+yi $ and substituting $ p $ into \eqref{eq:ab-orthS}, we have
$$
   y=\frac{(1-ab)d}{1+ab}i,
$$
and
$$
    p = d+yi
      = d-\frac{(1-ab)d}{1+ab}
      =\frac{2ab}{1+ab}d.
$$
Hence, we have the assertion.    
\end{proof}

In the same way, we obtain the following result.

\begin{cor}
For $ a,b\in S^1(0,1)\cap \mathbb{H}^2$, 
  ${\rm Re}(a-b)\neq 0$, and $ {\rm Im}(a-b)\neq0 $,
$$
  q=\dfrac{2ab}{1+ab}g,
$$
where
$$
  g=\frac1{1+ab}\Big(a+b+2{\rm sgn}\big({\rm Re}(a+b)\big)
            \sqrt{ab{\rm Im}(a){\rm Im}(b)}\Big).
$$    
\end{cor}


\section{Collinear points}

In this section, we prove that four points $u,s,m$ and $v$, naturally connected with the proof of the main
results, are in fact collinear. See Figure \ref{fig5}.

\noindent
\begin{lem} \label{sv} For $ a,b\in\mathbb{H}^2 $, let $v=LIS[a,a_*,b,b_*]$ and
          $s=LIS[a,b_*,b,a_*]$. Then, 
$$
  v=\frac{2ab-a\overline{a}-b\overline{b}-|a-b||a-\overline{b}|}
         {a+b-\overline{a}-\overline{b}}
   =\frac{2ab-|a|^2-|b|^2-|a-b||a-\overline{b}|}
         {2i \cdot{\rm Im}(a+b)}
$$
and
$$
  s=\frac{2ab-a\overline{a}-b\overline{b}+|a-b||a-\overline{b}|}
         {a+b-\overline{a}-\overline{b}}
   =\frac{2ab-|a|^2-|b|^2+|a-b||a-\overline{b}|}
         {2i \cdot {\rm Im}(a+b)}
$$
hold.
\end{lem}
\bigskip

\noindent
\begin{proof} From Lemma \ref{epF},
$$
  v=LIS[a,a_{\ast},b,b_{\ast}]=\frac{V_n(1)}{V_d(1)},
$$
where
\begin{align*}
  V_n(1)&=(a-b-\overline{a}+\overline{b})|a-b|^2|a-\overline{b}|^2
         +(a-b+\overline{a}-\overline{b})(2ab-a\overline{b}-\overline{a}b)
          |a-b||a-\overline{b}|\\
       &\qquad 
        -(a-b)(\overline{a}-\overline{b})(a+b-\overline{a}-\overline{b})
          (a\overline{a}-b\overline{b}),\\
 V_d(1)&=(a-b+\overline{a}-\overline{b})(a+b-\overline{a}-\overline{b})
         (|a-b||a-\overline{b}|-(a-b)(\overline{a}-\overline{b})).
\end{align*}
As $ |a-b||a-\overline{b}|^2
     =(a-b)(\overline{a}-\overline{b})(a-\overline{b})(\overline{a}-b) $,
we have
$$
  v=\frac{(a-b)(\overline{a}-\overline{b})(ab-\overline{a}\overline{b})
           -(2ab-a\overline{b}-\overline{a}b)|a-b||a-\overline{b}|}
         {(a+b-\overline{a}-\overline{b})
         \big((a-b)(\overline{a}-\overline{b})-|a-b||a-\overline{b}|\big)}.
$$
By multiplying the denominator and the numerator by
$ (a-b)(\overline{a}-\overline{b})+|a-b||a-\overline{b}| $, we have
$ v=V_n(2)/V_d(2) $, where
\begin{align*}
  V_n(2)&=-(2ab-a\overline{b}-\overline{a}b)|a-b|^2|a-\overline{b}|^2
         -(a-b)(\overline{a}-\overline{b})(a-\overline{a})(b-\overline{b})
           |a-b||a-\overline{b}| \\
        & \qquad
          +(a-b)^2(\overline{a}-\overline{b})^2(ab-\overline{a}\overline{b}),\\
  V_d(2)&=-(a+b-\overline{a}-\overline{b})|a-b|^2|a-\overline{b}|^2
         +(a-b)^2(\overline{a}-\overline{b})^2(a+b-\overline{a}-\overline{b}).
\end{align*}
By substituting 
$ |a-b||a-\overline{b}|^2
     =(a-b)(\overline{a}-\overline{b})(a-\overline{b})(\overline{a}-b) $ 
again,
we have
$$
  v=\frac{2ab-a\overline{a}-b\overline{b}-|a-b||a-\overline{b}|}
         {a+b-\overline{a}-\overline{b}}.
$$

\begin{figure}
    \centering
    \begin{tikzpicture}
    \draw (0,1.5) circle (0.7mm);
    \draw (2,3) circle (0.7mm);
    \draw (-2,0) circle (0.7mm);
    \draw (0.666,1.326) circle (0.7mm);
    \draw (0.666,2.321) circle (0.7mm);
    \draw (0.666,4.062) circle (0.7mm);
    \draw (0,-1.5) circle (0.7mm);
    \draw (2,-3) circle (0.7mm);
    \draw (-2,0) -- (2,3);
    \draw (0,1.5) -- (2,-3);
    \draw (2,3) -- (0,-1.5);
    \draw (-0.390,0) -- (0.666,4.062) -- (5.765,0) -- (-0.390,0);
    \draw (-0.390,0) -- (-2,0);
    \draw (0.666,2.321) -- (-2,0);
    \draw (0.666,4.062) -- (0.666,0);
    \draw (-0.390,0) circle (0.7mm);
    \draw (0.666,0) circle (0.7mm);
    \draw (1.535,0) circle (0.7mm);
    \draw (2.687,0) circle (0.7mm);
    \draw (5.765,0) circle (0.7mm);
    \draw (0,1.5) -- (5.765,0);
    \draw (2,3) -- (-0.390,0);
    \draw (1.535,1.535) circle (1.535cm);
    \draw (2.687,0) circle (3.077cm);
    \draw[dashed] (0,1.5) arc (60.8:-12.9:4.103);
    \draw[dashed] (2,3) arc (12.9:-60.8:4.103);
    \node[scale=1.3] at (0.36,1.6) {$a$};
    \node[scale=1.3] at (2,3.4) {$b$};
    \draw[fill=white,color=white] (0.9,2.45) -- (1.4,2.45) -- (1.4,2.15) -- (0.9,2.15) -- (0.9,2.45);
    \node[scale=1.3] at (1.1,2.35) {$m$};
    \node[scale=1.3] at (0.95,1.45) {$s$};
    \node[scale=1.3] at (0.95,4.15) {$v$};
    \node[scale=1.3] at (-0.7,0.3) {$a_*$};
    \node[scale=1.3] at (0.3,-0.2) {$u$};
    \node[scale=1.3] at (1.85,0.3) {$d$};
    \node[scale=1.3] at (6.1,0.3) {$b_*$};
    \node[scale=1.3] at (-2,-0.3) {$c$};
    \node[scale=1.3] at (2.8,-0.3) {$u_1$};
    \end{tikzpicture}
    \caption{The points $u,s=LIS[a, b_{\ast}, b, a_{\ast}],m,$ and $v$ are collinear. The point $s$ 
  is also the point of intersection
  of the angle bisectors of the triangle $\triangle(a,b,u).$}
    \label{fig5}
\end{figure}

Next, $ s $  is given by
$$
  s=LIS[a,b_{\ast},b,a_{\ast}].
$$
This point $ s $ can also be expressed in terms of 
$ a $ and $ b $ using exactly the same procedure.
\end{proof}

\bigskip


\begin{cor} \label{svunit}  
For $ a,b\in S^1(0,1) \cap\mathbb{H}^2$,
$$
  v=\frac{2ab-{\rm sgn}\big({\rm Re}(a-b)\big)(a-b)}{a+b},
 \qquad
  s=\frac{2ab+{\rm sgn}\big({\rm Re}(a-b)\big)(a-b)}{a+b}.
$$
\end{cor}
\bigskip

\begin{proof} For  $ a,b\in S^1(0,1)\cap\mathbb{H}^2 $,
note that
$$
  a_{\ast}={\rm sgn}\big({\rm Re}(a-b)\big),\qquad
  b_{\ast}=-{\rm sgn}\big({\rm Re}(a-b)\big)
$$
and we obtain $ s $ and $ v $ in the same way as above.
\end{proof}


\begin{rem}\label{trivobs}
(1) Looking at the triangle $\triangle(m, \overline{m},-m)$ we see that $L[m,-i]$ bisects
the angle at $m$ of this right triangle. Moreover, by \eqref{hmid2}, we see that
  $d$ is the hyperbolic midpoint in the metric $\rho_{\mathbb{B}^2}$
  of the points $0$ and $u$ if $a,b\in S^1(0,1)\cap\mathbb{H}^2$.

(2) Let ${\rm Im}(a)<{\rm Im}(b), \quad {\rm Re}(a)<{\rm Re}(b)<0, $
as in Figure 5. Considering the similar right triangles 
$\triangle(m,u,u_1)$ and $\triangle(m,c,u), c= LIS(a,b,0,1),$  it is easily seen
 that the point $c$ is the image of $u$ in the inversion in $S^1(u_1,r),
 r=|m-u_1|.$ Hence $S^1(u_1,r)$ is the Apollonian circle with the base
 points $c$ and $u$, passing through $a$ and $b$. Thus, clearly for all $z \in S^1(u_1,r)$, the following
 angles are equal:
 \[
\measuredangle(c,z,a_*) =  \measuredangle(a_*, z,u)\,.
 \]
\end{rem}

\begin{nonsec}\label{challenge}{\bf Challenges for symbolic computation.}\label{challenges}
In the case of Figure 5, let $h(z) = a_* + t^2/\overline{(z-a_*)}$ be the inversion in $S^1(a_*,t), t= |m-a_*|.$ Then it turns out that
$
h(a) = v \,, h(s) = b \,.
$
Using a similar inversion in $S^1(b_*,|m- b_*|)$, we obtain corresponding
facts for the point pairs $\{b,a\}$ and  $\{a,v\}\,.$ Therefore
\[
\rho_{\mathbb{H}^2}(a,s)= \rho_{\mathbb{H}^2}(b,v)\,, \quad
\rho_{\mathbb{H}^2}(b,s)= \rho_{\mathbb{H}^2}(a,v) \,.
\]
It seems to be difficult to verify these facts using symbolic computation.
We have discovered these equalities by numeric checking and outline
below a manual brute force checking of a part of these claims, omitting
some lengthy formulas.

For $a,b \in \mathbb{H}^2, {\rm Re}(a-b)\neq 0 $,
let
\begin{equation}
v=\frac{2ab-a\overline{a}-b\overline{b}-|a-b||a-\overline{b}|}
       {a+b-\overline{a}-\overline{b}},\qquad
s=\frac{2ab-a\overline{a}-b\overline{b}+|a-b||a-\overline{b}|}
       {a+b-\overline{a}-\overline{b}},
\end{equation}
by Lemma \ref{sv}.

Then,
\begin{align*}
  |v-a|^2&|s-\overline{b}|^2-|s-b|^2|v-\overline{a}| \\
     &=(v-a)(\overline{v}-\overline{a})(s-\overline{b})(\overline{s}-b)
        -(s-b)(\overline{s}-\overline{b})(v-\overline{a})(\overline{v}-a)\\
     &= \mbox{ brute force calculations....} \\
     &=\dfrac{1}{(a+b-\overline{a}-\overline{b})^3}
       \Big((a+b-\overline{a}-\overline{b})|a-b||a-\overline{b}|
             -(a-\overline{b})(\overline{a}-b)(a-b-\overline{a}+\overline{b})
       \Big) \\
     & \quad
       \times\Big(|a-b|^2|a-\overline{b}|^2-(a-b)(\overline{a}-\overline{b})
                   (a-\overline{b})(\overline{a}-b)\Big).
\end{align*}
The last factor can be written as
$$
    |a-b|^2|a-\overline{b}|^2-(a-b)(\overline{a}-\overline{b})
                   (a-\overline{b})(\overline{a}-b)
    = |a-b|^2|a-\overline{b}|^2-|a-b|^2|a-\overline{b}|^2=0.
$$
Hence, we have $ |v-a||s-\overline{b}|=|s-b||v-\overline{a}| $ and
$$
    \th \frac{\rho(v,a)}{2}=\frac{|v-a|}{|v-\overline{a}|}
    =\frac{|s-b|}{|s-\overline{b}|}= \th \frac{\rho(s,b)}{2}.
$$
\end{nonsec}




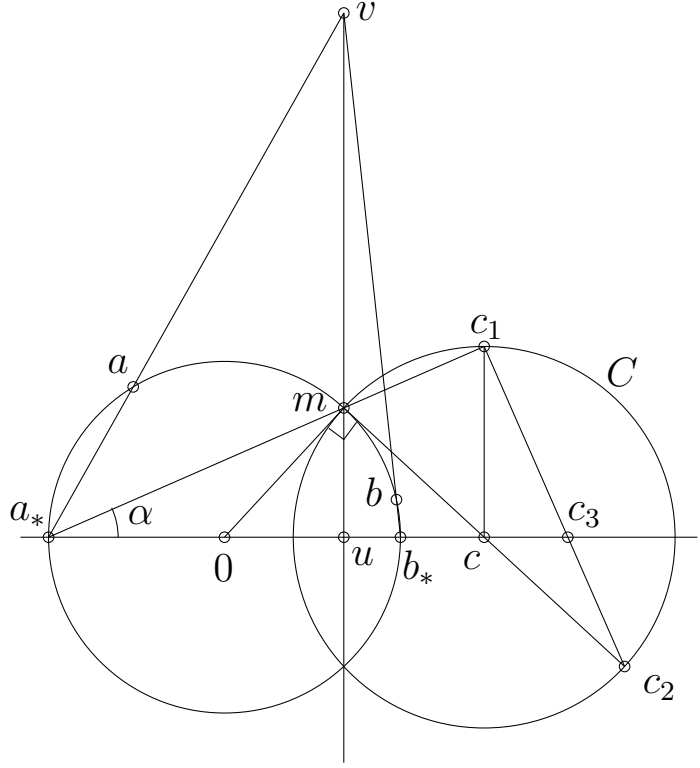
\begin{figure}
    \centering
    \begin{tikzpicture}
    \draw (-0.5,2) circle (0.7mm);
    \draw (3,0.5) circle (0.7mm);
    \draw (2.3,1.720) circle (0.7mm);
    \draw (2.3,6.975) circle (0.7mm);
    \draw (2.3,0) circle (0.7mm);
    \draw (0.714,0) circle (0.7mm);
    \draw (4.166,0) circle (0.7mm);
    \draw (4.166,2.538) circle (0.7mm);
    \draw (6.033,-1.720) circle (0.7mm);
    \draw (-1.625,0) circle (0.7mm);
    \draw (3.054,0) circle (0.7mm);
    \draw (5.278,0) circle (0.7mm);
    \draw (2.3,6.975) -- (-1.625,0);
    \draw (2.3,6.975) -- (2.3,-3);
    \draw (2.3,6.975) -- (3.054,0);
    \draw (4.166,2.538) -- (4.166,0);
    \draw (4.166,2.538) -- (6.033,-1.720);
    \draw (4.166,2.538) -- (-1.625,0);
    \draw (2.3,1.720) -- (6.033,-1.720);
    \draw (0.714,0) -- (2.3,1.720);
    \draw (-2,0) -- (7,0);
    \draw (0.714,0) circle (2.339cm);
    \draw (4.166,0) circle (2.538cm);
    \draw (-0.7,0) arc (0:25:0.925cm);
    \draw (2.1,1.45) -- (2.3,1.3) -- (2.47,1.53);
    \node[scale=1.3] at (-0.7,2.3) {$a$};
    \node[scale=1.3] at (2.7,0.6) {$b$};
    \node[scale=1.3] at (1.85,1.8) {$m$};
    \node[scale=1.3] at (0.7,-0.4) {$0$};
    \node[scale=1.3] at (2.6,7) {$v$};
    \node[scale=1.3] at (2.55,-0.25) {$u$};
    \node[scale=1.3] at (-1.9,0.3) {$a_*$};
    \node[scale=1.3] at (3.3,-0.4) {$b_*$};
    \node[scale=1.3] at (4,-0.3) {$c$};
    \node[scale=1.3] at (4.2,2.8) {$c_1$};
    \node[scale=1.3] at (5.5,0.3) {$c_3$};
    \node[scale=1.3] at (6.5,-2) {$c_2$};
    \node[scale=1.3] at (-0.4,0.3) {$\alpha$};
    \node[scale=1.3] at (6,2.2) {$C$};
    \end{tikzpicture}
    \caption{Here $a,b \in S^1(0,1) \cap \mathbb{H}^2$, $v={\rm LIS}[b_*,b,a,a_*],$ $u$ is the projection of $v$ on the real axis and $m= S^1(0,1)\cap [u,v].$ The circle $C$ is  $S^1(c,|c-m|)$ for $c={\rm LIS}[a,b,0,1]$. The $ \angle(O,m,c) $ is a right angle.}
    \label{fig6}
\end{figure}

\bigskip

Fix the notation for the next lemma as follows:
For $ a,b\in S^1(0,1)\cap \mathbb{H}^2 $,
let $ c $ be the intersection point of the real axis and
the line passing through $  a $ and  $ b $,
$ v $ the intersection point of the line passing through $ a, a_{\ast}=-1 $
and the line passing through $ b,b_{\ast}=1 $,
$ u $ the foot of the perpendicular from $ v $ to the real axis, and
$ m $ the intersection point of $ S^1(0,1) $ and the line passing through
$ u,v $ satisfying $ m\in\mathbb{H}^2 $.
Moreover, let $ C $ be the circle with center at $ c $ and radius $ |c-m| $,
and $ c_1 $ the intersection point of the circle $ C $ and
the line perpendicular to the real axis passing through $ c $
satisfying $ c_1\in\mathbb{H}^2 $.

That is,
\begin{align*}
   c &= LIS[a,b,0,1], \quad  v = LIS[a,a_*,b,b_*],\\
  u &= LIS[a,b^{\ast},b,a^{\ast}], \quad  m  \in S^1(0,1)\cap \mathbb{H}^2\cap L[u,v],\\
     C &=S^1(c,|c-m|),\quad c_1 = C\cap {\mathbb{H}^2}\cap L[c,c+i] .
     \end{align*}

See Figure \ref{fig6}.

\bigskip

The above $m$ gives the hyperbolic midpoint of $a$ and $b$ (see Remark \ref{rem_4.5}).

\begin{lem}
If $ a,b\in S^1(0,1)\cap \mathbb{H}^2 $, $c= LIS[a,b,0,1]$, and $m$ is the hyperbolic midpoint of $a$ and $b$, then
$ \angle(O,m,c) $ is a right angle.    
\end{lem}
\begin{proof}
Let $ \alpha=\angle(O,a_*,m) $, $ c_2 $ the intersection point
of a line parallel to $ L[m,b_*] $ through $ c_1 $ and the circle $ C$,
and $ c_3 $ the intersection point of $ L[c_1,c_2] $ and the real axis.

From Thales' theorem, angles $ \angle(a_*,m,b_*) $ and 
$ \angle(m,c_1,c_2) $ are right angles.
The triangles $ \triangle(m,a_*,b_*)$, $ \triangle (u,b_*,m) $ 
and $ \triangle (c,c_3,c_1) $
are similar, so $ \angle(u,m,b_*)=\angle(c,c_1,c_3)=\alpha $.
The triangle $ \triangle(c_1,c,c_2) $ is an isosceles triangle, 
so $ \angle(c,c_2,c_1)=\angle(c,c_1,c_2)=\alpha $,
and its central angle, $ \angle(m,c,c_1) $, is $ 2\alpha $.
Therefore, we have $ \angle(m,c,O)=\frac{\pi}2-2\alpha $.

Here, $ \angle (m,O,b_*) $ is the central angle of
$ \angle(m,a_*,b_*) $, so $ \angle (m,O,b_*)=2\alpha $.

Consider the triangle $ \triangle(O,m,c) $.
Since $ \angle(m,c,O)=\frac{\pi}2-2\alpha $ and 
$ \angle (m,O,c)=2\alpha $, $ \triangle(O,m,c) $ is a
right triangle with $ \angle(O,m,c)=\frac{\pi}2 $.

\medskip

Hence, we have the assertion, and
the two circles $ S^1(0,1) $ and $ C $ are orthogonal to each other.
\end{proof}

\bigskip



\begin{rem}\label{rem_4.5}
From the proof of Lemmas \ref{uformula} and \ref{sv}, we have
\begin{align*}
  u &=\frac{{\rm Im}(ab)}{{\rm Im}(a+b)}, \\
  v &=\frac{2ab-|a|^2-|b|^2-|a-b||a-\overline{b}|}    
         {2i \cdot{\rm Im}(a+b)}\\
    & = \frac{{\rm Im}(ab)}{{\rm Im}(a+b)}
        -i\frac{{\rm Re}(ab)-|a|^2-|b|^2-|a-b||a-\overline{b}|}
               {2{\rm Im}(a+b)}, \\
  s &=\frac{2ab-|a|^2-|b|^2+|a-b||a-\overline{b}|}
         {2i \cdot {\rm Im}(a+b)} \\
    & = \frac{{\rm Im}(ab)}{{\rm Im}(a+b)}
        -i\frac{{\rm Re}(ab)-|a|^2-|b|^2+|a-b||a-\overline{b}|}
               {2{\rm Im}(a+b)}, \\
  m &=\frac{{\rm Im}(ab)+i|a-\overline{b}|\sqrt{{\rm Im}(a){\rm Im}(b)}}
          {{\rm Im}(a+b)} \\
    & =\frac{{\rm Im}(ab)}{{\rm Im}(a+b)}
           +i\frac{|a-\overline{b}|\sqrt{{\rm Im}(a){\rm Im}(b)}}
          {{\rm Im}(a+b)}.
\end{align*}

The real part of the four points $ u,v,s,m $ are all 
$ ({\rm Im}(ab))/({\rm Im}(a+b)) $.
Therefore, these points lie on the line perpendicular to
the real axis passing through $ u $.
\end{rem}


\bigskip


\section{ Formula for $v_{\mathbb{H}^2}(a,b),$ $a,b\in S^1(0,1)\cap \mathbb{H}^2$}
  
We give a formula for $v_{\mathbb{H}^2}(a,b)$ 
with $a,b\in S^1(0,1)\cap \mathbb{H}^2$ in four steps.

\begin{nonsec}{\bf Four steps.} \label{4steps}
Figure \ref{fig7} illuminates the steps. Below, $a,b\in S^1(0,1)\cap \mathbb{H}^2.$
\begin{itemize}
\item[(1)] See formula for $u=LIS[a, \overline{b}, \overline{a},b]$ from Lemma \ref{uformula}.
\item[(2)] The circle through $a, b, d  = u/(1+ \sqrt{1-|u|^2})$ is tangent to the segment
$\,(-1,1)\,$ at $d$, see Remark \ref{dishm}.
\item[(3)] Let $k\in (-1,1),$ $$ya\in  L[a,k]\cap S^1(0,1) ;
 \quad yb \in L[b,k]\cap S^1(0,1)$$
with ${\rm Im}(ya)<0, {\rm Im}(yb)<0. $ If $w=k/(1+\sqrt{1-|k|^2}),$ then
$T_w(ya)+T_w(a) = 0= T_w(yb)+T_w(b)$ and hence
$\rho_{\mathbb{H}^2}(a,b)=\rho_{\mathbb{H}^2}(\overline{ya},\overline{yb}).$ Moreover, if $k=d$, then
$$
v_{\mathbb{H}^2}(a,b)=\measuredangle(a,d,b)\,.
$$
See Lemma \ref{4steps3}.
\item[(4)] Let $u=(1+ a\cdot b)/(a+b)\,,  d={u}/({1+\sqrt{1-|u|^2}}) $ and
$\theta =\measuredangle(a,d,b).$ Then $\theta=v_{\mathbb{H}^2}(a,b)$,  
\[
\sin(\theta)=  \frac{(1+\sqrt{1-|u|^2})\, t \sqrt{1-t^2}}{\sqrt{1- |u|^2 t^2}}\equiv T\,,
\quad  t= {\rm th}(\rho_{\mathbb{H}^2}(a,b)/2)\,.
\]
Moreover, $\theta={\rm arcsin} T$ if $1> t^2(1+\sqrt{1-|u|^2})$ and $\theta=\pi-{\rm arcsin} T$ if \newline $1< t^2(1+\sqrt{1-|u|^2})\,.$
 See Lemma \ref{4steps4} and Theorem \ref{my4step}.
\end{itemize}
\end{nonsec}

\bigskip

\begin{prop}\label{trivcomp}
For $a,b \in S^1(0,1)\cap \mathbb{H}^2$, let
\[
t= {\rm th} \frac{\rho_{\mathbb{H}^2}(a,b)}{2}= \frac{|a-b|}{|a-\overline{b}|}\,.
\]
Then
\[
\rho_{\mathbb{H}^2}(a,b) = \rho_{\mathbb{H}^2}(t+i \sqrt{1-t^2}, -t+i \sqrt{1-t^2}).
\]
\end{prop}

\begin{proof}
The proof follows from \eqref{rhoH}.
\end{proof}

\begin{prop}\label{yayb}
Let $a,b \in S^1(0,1)\cap \mathbb{H}^2$ and $k\in (-1,1)$ and $ya\in S^1(0,1)\cap L[a,k],$
 $yb\in S^1(0,1)\cap L[b,k]$ with ${\rm Im}(ya)<0$ and ${\rm Im}(yb)<0.$ Then
\[
\rho_{\mathbb{H}^2}(a,b)= \rho_{\mathbb{H}^2}(\overline{ya},\overline{yb})\,.
\]
\end{prop}

\begin{proof} The two hyperbolic lines with end points $a$ and $ya$, resp. $b$ and $yb,$
intersect the real axis at the hyperbolic midpoint $w= k/(1+\sqrt{1-|k|^2})$
of $0$ and $k.$ The M\"obius
transformation $T_w$ maps these two hyperbolic lines onto two diameters of the unit disk and
it is now clear that $T_w(a)+T_w(ya)=0 $ and  $T_w(b)+T_w(yb)=0 $ and hence the proposition
follows.
\end{proof}

\begin{lem} \label{4steps3}
The step (3) of the four steps \ref{4steps} holds.
\end{lem}

\begin{proof} See Figure \ref{fig7}(A). The proof follows from Proposition 5.3. Choosing $k=d$ in 5.3, we see that in the case
${\rm Re}(a)< {\rm Re}(b)$,
\[
ya=t -i\,\sqrt{1-t^2}, \quad  yb=-t -i\,\sqrt{1-t^2}\,.
\]
\end{proof}


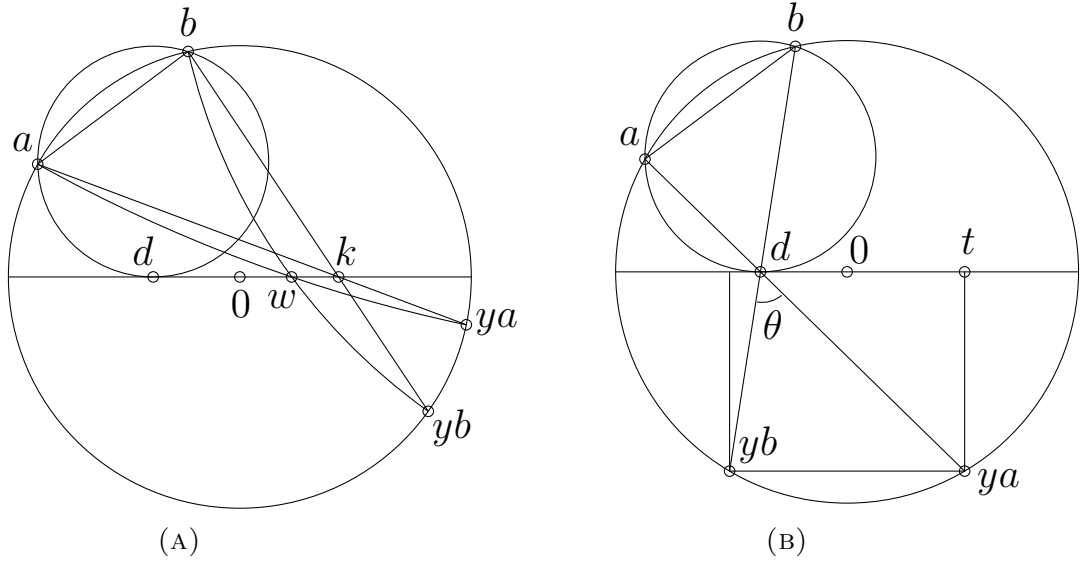
\begin{figure}
    \centering
    \begin{subfigure}[b]{0.3\textwidth}
    \centering
    \begin{tikzpicture}
    \draw (0,1.5) circle (0.7mm);
    \draw (2,3) circle (0.7mm);
    \draw (1.535,0) circle (0.7mm);
    \draw (2.687,0) circle (0.7mm);
    \draw (4,0) circle (0.7mm);
    \draw (3.376,0) circle (0.7mm);
    \draw (5.698,-0.636) circle (0.7mm);
    \draw (5.192,-1.788) circle (0.7mm);
    \draw (0,1.5) -- (2,3);
    \draw (-0.390,0) -- (5.765,0);
    \draw (0,1.5) -- (5.698,-0.636);
    \draw (2,3) -- (5.192,-1.788);
    \draw (2.687,0) circle (3.077cm);
    \draw (0,1.5) arc (240.832:258.055:20.323cm);
    \draw (2,3) arc (192.907:234.472:8.109cm);
    \draw (1.535,1.535) circle (1.535cm);
    \node[scale=1.3] at (-0.2,1.8) {$a$};
    \node[scale=1.3] at (2,3.4) {$b$};
    \node[scale=1.3] at (1.4,0.35) {$d$};
    \node[scale=1.3] at (2.7,-0.35) {$0$};
    \node[scale=1.3] at (3.25,-0.3) {$w$};
    \node[scale=1.3] at (4.1,0.35) {$k$};
    \node[scale=1.3] at (6.1,-0.6) {$ya$};
    \node[scale=1.3] at (5.5,-2) {$yb$};
    \end{tikzpicture}
    \caption{}
    \end{subfigure}
    \hspace{3cm}
    \begin{subfigure}[b]{0.3\textwidth}
    \centering
    \begin{tikzpicture}
    \draw (0,1.5) circle (0.7mm);
    \draw (2,3) circle (0.7mm);
    \draw (1.535,0) circle (0.7mm);
    \draw (4.250,-2.651) circle (0.7mm);
    \draw (1.125,-2.651) circle (0.7mm);
    \draw (4.250,0) circle (0.7mm);
    \draw (2.687,0) circle (0.7mm);
    \draw (0,1.5) -- (2,3);
    \draw (0,1.5) -- (4.250,-2.651);
    \draw (2,3) -- (1.125,-2.651);
    \draw (-0.390,0) -- (5.765,0);
    \draw (1.125,0) -- (1.125,-2.651) -- (4.250,-2.651) -- (4.250,0);
    \draw (2.687,0) circle (3.077cm);
    \draw (1.535,1.535) circle (1.535cm);
    \draw (1.49,-0.4) arc (260:310:0.4cm);
    \node[scale=1.3] at (-0.2,1.8) {$a$};
    \node[scale=1.3] at (2,3.4) {$b$};
    \node[scale=1.3] at (1.8,0.35) {$d$};
    \node[scale=1.3] at (1.7,-0.7) {$\theta$};
    \node[scale=1.3] at (4.7,-2.8) {$ya$};
    \node[scale=1.3] at (2.85,0.3) {$0$};
    \node[scale=1.3] at (4.3,0.4) {$t$};
    \node[scale=1.3] at (1.5,-2.3) {$yb$};
    \end{tikzpicture}
    \caption{}
    \end{subfigure}
    \caption{(A) For all $a,b \in S^1(0,1)\cap \mathbb{H}^2$ and $k\in (-1,1),$ $\rho_{\mathbb{H}^2}(a,b)=\rho_{\mathbb{H}^2}(\overline{ya},\overline{yb}).$ Moreover, if $w=k/(1+ \sqrt{1-|k|^2}),$ then $T_w(ya)+T_w(a)=T_w(yb)+T_w(b)=0.$ 
    \newline
    (B) By choosing $k=d$ in Figure \ref{fig7}(A), we have ${\rm Im}(ya)= {\rm Im}(yb).$}
    \label{fig7}
\end{figure}

\begin{lem} \label{4steps4}
Let $a,b\in S^1(0,1)\cap \mathbb{H}^2$ and $u=(1+ab)/(a+b),\, d= u/(1+\sqrt{1-|u|^2}),$ 
and let $ya,yb$ be as  in Lemma \ref{4steps3} and $\theta=\measuredangle(ya,d,yb).$ Then $\theta= v_{\mathbb{H}^2}(a,b)$ and
\[
\sin(\theta)  =\frac{(1+\sqrt{1-|u|^2})\, t \sqrt{1-t^2}}{\sqrt{1- |u|^2 t^2}}\,,\quad
t= {\rm th} \frac{\rho_{ \mathbb{H}^2}(a,b)}{2} \,.
\]
 \end{lem}

\begin{proof} See Figure \ref{fig7}(B). The area of the triangle $\triangle(ya,d,yb)$ is
$t \sqrt{1-t^2}$ and, by trigonometry, also equal to
\begin{equation} \label{triarea}
\frac{1}{2}\, |ya-d||yb-d| \sin \theta =t \sqrt{1-t^2}\,.
\end{equation}
By the Pythagorean theorem,
\[ |ya-d|^2 = 1-t^2 + (t-d)^2 = 1-2td +d^2\,, \]
\[ |yb-d|^2 = 1-t^2 + (t+d)^2 = 1+2td +d^2\,. \]
Because $ d= u/(1+\sqrt{1-|u|^2}),$ we see that

\begin{equation*}
  \begin{cases}
  {\displaystyle 1+d^2 = \frac{2}{1+ \sqrt{1-|u|^2}}}&\\[3mm]
{\displaystyle (|ya-d||yb-d|)^2 = (1+d^2)^2- 4t^2 d^2 = 
\frac{4 (1- u^2 t^2)}{(1+ \sqrt{1-|u|^2})^2}}\,,&
\end{cases}
\end{equation*}
which by \eqref{triarea} yields the desired formula.
\end{proof}

\begin{thm} \label{my4step} For $a,b\in S^1(0,1)\cap \mathbb{H}^2$
let $u=(1+ab)/(a+b)\,.$ Then
$$
    v_{\mathbb{H}^2}(a,b)
    =\begin{dcases}
      \arcsin T  &  {\rm if} \ 1\geq (1+\sqrt{1-|u|^2}) t^2,\\
      \pi - \arcsin T &  {\rm if} \ 1 < (1+\sqrt{1-|u|^2}) t^2\, ,
   \end{dcases}
$$
where
\[
T= \frac{(1+\sqrt{1-|u|^2})\, t \sqrt{1-t^2}}{\sqrt{1- |u|^2 t^2}}\,,\quad
t= {\rm th} \frac{\rho_{ \mathbb{H}^2}(a,b)}{2} \,.
\]
\end{thm}

\begin{proof}
By the choice of the point $d$ in Lemma \ref{4steps4}, $v_{\mathbb{H}^2}(a,b) =\theta.$
By the proof of Lemma \ref{4steps4}, we see that $\theta \in (\pi/2,\pi)$ if and only if
\[  |ya-d|^2 +|yb-d|^2 < |ya-yb|^2 \Leftrightarrow \]
\[ 1+d^2 < 2t^2  \Leftrightarrow \]
\[   1< (1+\sqrt{1-|u|^2}) t^2 \]
and, in this case $\theta \in (\pi/2,\pi)$, we have $\theta= \pi- \arcsin T \,.$ In the case
$\theta \in(0, \pi/2),$ $\theta= \arcsin T.$
\end{proof}

The next proposition gives an alternative way of writing the right hand side of
Lemma \ref{4steps4}.

\begin{prop} \label{sth2}
For $ a,b\in S^1(0,1)\cap\mathbb{H}^2 $,
$ u=LIS[a,\overline{b},\overline{a},b] $, and
$ t={\rm th}\big(\rho_{\mathbb{H}^2}(a,b)/2\big) $, we have
$$
  \left|\frac{(1+\sqrt{1-|u|^2})t\sqrt{1-t^2}}{\sqrt{1-u^2t^2}}\right|
  =\frac{|a-b|}{|ab-1|}\left(
    \frac12|a+b|+\sqrt{{\rm Im}(a)\cdot {\rm Im}(b)}\right).
$$
\end{prop} 

\bigskip

\noindent
\begin{proof} \quad
Because $\overline{b} =1/b ,$ by \eqref{rhoH} we see that 
$$ t= {\rm th}\frac{\rho_{\mathbb{H}^2}(a,b)}{2}
    =\left|\frac{a-b}{a-\overline{b}}\right|
    =\left|\frac{a-b}{a-\frac1b}\right|
    =\frac{|b||a-b|}{|ab-1|}=\frac{|a-b|}{|ab-1|}.
$$      
Moreover, from Lemma \ref{uformula}(2),
$$
  u=\frac{1+ab}{a+b}.
$$
Hence, we have
\begin{align*}
  \sqrt{1-t^2} &= \sqrt{\frac{(a^2-1)(b^2-1)}{(ab-1)^2}},\\
  \sqrt{1-|u|^2} &= \sqrt{\frac{-(a^2-1)(b^2-1)}{(a+b)^2}},\\
  \sqrt{1-u^2t^2} &= \sqrt{\frac{4ab(a^2-1)(b^2-1)}{(a+b)^2(ab-1)^2}},
\end{align*}
and
\begin{align*}
  &\left|\frac{(1+\sqrt{1-|u|^2})t\sqrt{1-t^2}}{\sqrt{1-u^2t^2}}\right|\\
  &= \sqrt{\frac{(a+b)^2(ab-1)^2}{4ab(a^2-1)(b^2-1)}}
     \left(1+\sqrt{\frac{-(a^2-1)(b^2-1)}{(a+b)^2}}\right)
     \left|\frac{a-b}{ab-1}\right|
     \sqrt{\frac{(a^2-1)(b^2-1)}{(ab-1)^2}}\\
  &= \frac{|a-b|}{|ab-1|}\frac12\sqrt{(a+b)(\frac1a+\frac1b)}
     \left(1+\sqrt{\frac{-(a-\frac1a)(b-\frac1b)}
                        {(a+b)(\frac1a+\frac1b)}}\right)   \\
  &=\frac{|a-b|}{|ab-1|}\left(\frac12|a+b|
        +\sqrt{ {\rm Im}(a)\cdot {\rm Im}(b)}\right).
\end{align*}  
\end{proof}


\medskip

\begin{nonsec}{\bf Proof of Theorem \ref{my1}.} The proof follows from Lemma
\ref{4steps4} and Proposition \ref{sth2}. \hfill $\square$
\end{nonsec}

\medskip

\begin{rem} (1) 
For $ a,b\in\mathbb{H}^2$, we have 
  $$
      \frac12|a+b|\geq \sqrt{ {\rm Im}(a)\cdot {\rm Im}(b)}\,.\qquad
     $$
Indeed, since $  {\rm Im}(a), {\rm Im}(b)>0 $,
\[  \sqrt{ {\rm Im}(a)\cdot {\rm Im}(b)}
\leq\frac12\big( {\rm Im}(a)+ {\rm Im}(b)\big)
=\frac12 {\rm Im}(a+b)\leq \frac12|a+b|.
\]

(2) By (1), we see that, for $  a,b\in S^1(0,1)\cap \mathbb{H}^2$, 
the expression $T$
in Theorem \ref{my4step} and Proposition \ref{sth2} admits the following estimate
\[ 2 t \sqrt{ {\Im}(a) {\Im}(b)} \le T \le t |a+b| .\]

(3) For instance, we have
\[
v_{\mathbb{H}^2}(e^{\varphi i},e^{3\varphi i})={\rm arcsin} \frac{1+ \sqrt{\sqrt{2}-1}}{2}\approx 0.964558\,, \quad \varphi=\pi/8.
\]
\end{rem}


\begin{table}[ht]
    \centering
    \begin{tabular}{|l|l|l|l|}
        \hline
        $a$ & $b$ & $p$ & $v_{\mathbb{H}^2}(a,b)$\\
        \hline
        $2i$ & $-1+i$ & $i$ & $\pi/4$\\
        \hline
        $2i$ & $-3+i$ & $-6+15i+(2-6i)\sqrt{5}$ & $\pi-{\rm arctan}(9+4\sqrt{5})$\\
        \hline
        $2i$ & $1.2+4i$ & $-2.09848+2.10091i$ & $0.588469$\\
        \hline
        $2i$ & $1+3i$ & $-2+5i+(2-2i)\sqrt{3}$ & ${\rm arctan}((5+4\sqrt{3})/23)$\\
        \hline
        $2i$ & $-2+i$ & $-4+15i/2+(1-2i)\sqrt{10}$ & ${\rm arctan}(3/2+\sqrt{5/2})$\\
        \hline
        $2i$ & $-4+i$ & $-8+51i/2+(1-4i)\sqrt{34}$ & $\pi-{\rm arctan}((6+\sqrt{34})/4)$\\
        \hline
        $2i$ & $-3+3i$ & $6+25i-(2+6i)\sqrt{15}$ & ${\rm arctan}(1+4/\sqrt{15})$\\
        \hline
    \end{tabular}
    \caption{Accurate values of the point $p$ defined as in Table \ref{t1} and $v_{\mathbb{H}^2}(a,b)$ for example point pairs $a,b$.}
    \label{t3}
\end{table}

\bigskip
\begin{prop} 
For  $ a,b\in S^1(0,1)\cap\mathbb{H}^2 $, let
 the point $d$ be as in Lemma \ref{4steps4}, and
let $ \theta=\measuredangle (a,d,b) $.
Then 
$$
\left.\sin\theta=\left|{\rm Im}\Big(\frac{a-d}{b-d}\Big)\right|\middle/\Big|\frac{a-d}{b-d}\Big|.\right.
$$
\end{prop}
\bigskip

\begin{proof} Note that $ 0<\theta<\pi $ and $ \sin\theta>0 $, as 
$ a,b\in\mathbb{H}^2 $. Since 
$$ 
   \theta=\arg\frac{a-d}{b-d},
$$
we can write $ (a-d)/(b-d)$ as
$$
   \frac{a-d}{b-d}=\left|\frac{a-d}{b-d}\right|(\cos\theta+i\sin\theta)
$$ 
and the equality is obtained.
\end{proof}


\bigskip






%


\section{H\"older continuity and $v_{\mathbb{H}^2}$}

In this final section, we study how the visual angle metric $v_{\mathbb{H}^2}$
is distorted under quasiregular mappings. Our result is similar to a result for
$v_{\mathbb{B}^2}$ in \cite{wv}.


Let us recapitulate some fundamental facts
about $K$-quasiregular mappings. The definition of these mappings  is given in \cite[pp. 288-289]{hkv}. 
For the proof of the following result, we need some
properties of an increasing homeomorphism $\varphi_{K,2}:[0,1]\to[0,1]$, \cite[(9.13), p. 167]{hkv} 
\begin{align*}
\varphi_{K,2}(r)=\frac{1}{\gamma_2^{-1}(K\gamma_2(1\slash r))} =\mu^{-1}(\mu(r)/K),\quad
0<r<1,\,K>0.
\end{align*}
This function is the special function of the Schwarz
lemma in \cite[Thm 16.2]{hkv}, crucial for the sequel.
The Gr\"otzsch capacity, a decreasing homeomorphism $\gamma_2:(1,\infty)\to (0,\infty)$, is defined as \cite[(7.18), p. 122]{hkv}
\begin{align*}
\gamma_2(1/r)=\frac{2\pi}{\mu(r)},\quad \mu(r)=\frac{\pi}{2}\frac{\K(\sqrt{1-r^2})}{\K(r)},\quad
\K(r)=\int^1_0 \frac{dx}{\sqrt{(1-x^2)(1-r^2x^2)}}
\end{align*}
with $0<r<1$. Another special function we need is
\begin{equation} \label{myeta}
\displaystyle
\eta_{K,2}(t) = \frac{1-\varphi_{1/K,2}(1/\sqrt{1+t})^2}{\varphi_{1/K,2}(1/\sqrt{1+t})^2} = \left( \frac{\varphi_{K,2}(\sqrt{t/(1+t)})}{\varphi_{1/K,2}(1/\sqrt{1+t})}\right)^2\,, \quad t > 0\,.
\end{equation}

The following constant will be important \cite[(10.4), p. 203]{avv}
\begin{align*}
\lambda(K)=\left(\frac{\varphi_{K,2}(1/\sqrt{2})}{\varphi_{1/K,2}(1/\sqrt{2})}\right)^2
=\eta_{K,2}(1)\,,\quad K>1.   
\end{align*}
Trivially, $\lambda(1)=1$ and, 
 by \cite[Thm 10.35, p. 219]{avv},
$\lambda(K)<e^{\pi(K-1/K)}$ for $K>1$. The function
$\varphi_{K,2}$ satisfies many identities and inequalities
as shown in \cite[Section 10]{avv}. 

\begin{lem} \label{rv68} For $K\ge 1, r \in (0,1),$
we have
 \begin{equation} \label{phibd}
\frac{\varphi_{K,2}(r)}{\sqrt{1- \varphi_{K,2}(r)^2}} 
= \sqrt{\eta_{K,2}\left(\frac{r^2}{1-r^2}\right)}
\le \lambda(K)^{1/2}
\max \left\{ ( \frac{r}{\sqrt{1-r^2}} )^{1/K},
 ( \frac{r}{\sqrt{1-r^2}} )^K \right\} \,.
\end{equation} 
In particular, if $r= {\rm th }(t/2), t>0,$ we have
 \begin{equation} \label{etabd}
 \frac{\varphi_{K,2}({\rm th }(t/2))}{\sqrt{1- \varphi_{K,2}({\rm th }(t/2))^2}} \le 
 \lambda(K)^{1/2}
\max \left\{ ( {\rm sh}(t/2)  )^{1/K},
 ({\rm sh}(t/2)  )^K \right\} \,.
 \end{equation}
\end{lem}

\begin{proof}
By  \cite[ 10.24]{avv}  it follows that for $K>1, r \in (0,1),$
\[
\sqrt{\eta_{K,2}(\frac{r^2}{1- r^2})} \le  \lambda(K)^{1/2}
\max \left\{ ( \frac{r}{\sqrt{1-r^2}} )^{1/K},
 ( \frac{r}{\sqrt{1-r^2}} )^K \right\} \,,
\]
which proves \eqref{phibd}. Setting $r= {\rm th }(t/2), t>0,$ in \eqref{phibd} we have $r^2/(1-r^2)= {\rm sh}^2 (t/2)$ and
 \eqref{etabd} follows immediately. 
\end{proof}

\begin{nonsec}{\bf Proof of Theorem \ref{thm_holder}.}
By \eqref{klvw319}  for $a,b\in\mathbb{H}^2$,
\[
\displaystyle
\tan\left(\frac{v_{\mathbb{H}^2}(f(a),f(b))}{2}\right)
\le {\rm sh} \frac{\rho_{\mathbb{H}^2}(f(a),f(b))}{2}
= \frac{{\rm th}  \frac{\rho_{\mathbb{H}^2}(f(a),f(b))}{2}}{\sqrt{1- {\rm th}^2  \frac{\rho_{\mathbb{H}^2}(f(a),f(b))}{2} }}
\]
and further, by  \eqref{etabd} and the Schwarz lemma \cite[Thm 16.2]{hkv},
\begin{align*}
\displaystyle
 \frac{{\rm th}  \frac{\rho_{\mathbb{H}^2}(f(a),f(b))}{2}}{\sqrt{1- {\rm th}^2  \frac{\rho_{\mathbb{H}^2}(f(a),f(b))}{2} }} \le
 \frac{\varphi_{K,2}( {\rm th}\frac{\rho_{\mathbb{H}^2}(a,b)}{2})}
 {\sqrt{1- \varphi_{K,2}({\rm th} \frac{\rho_{\mathbb{H}^2}(a,b)}{2})^2}}
\le \\ \lambda(K)^{1/2}
\max\left\{\left({\rm sh}\frac{\rho_{\mathbb{H}^2}(x,y)}{2}\right)^K,\left({\rm sh}\frac{\rho_{\mathbb{H}^2}(x,y)}{2}\right)^{1/K}\right\}\,.
\end{align*}
Finally, the above inequalities together with \eqref{klvw319} yield the desired conclusion

\begin{align*}
\tan\left(\frac{v_{\mathbb{H}^2}(f(a),f(b))}{2}\right)\leq
\lambda(K)^{1/2}
\max\left\{\left(\tan(v_{\mathbb{H}^2}(a,b))\right)^K,\left(\tan(v_{\mathbb{H}^2}(a,b))\right)^{1/K}\right\}\,.
\end{align*}    
\hfill $\square$
\end{nonsec}

\begin{rem}
For $K=1$, Theorem \ref{thm_holder} yields for  $v_{\mathbb{B}^2}(a,b) < \pi/2$
\begin{align*}
\tan\left(\frac{v_{\mathbb{H}^2}(f(a),f(b))}{2}\right)\leq\tan(v_{\mathbb{H}^2}(a,b)),   
\end{align*} 
which is sharp for M\"obius transformations by \ref{2.19}.

\end{rem}

\end{document}